\documentclass[11pt,reqno]{amsart}

\usepackage{amsmath,amssymb,amsthm}
\usepackage{tikz}
\usetikzlibrary{decorations.pathmorphing}
\usepackage{graphicx, subcaption}
\usepackage{hyperref}
\usepackage[margin=2.8cm]{geometry}
\usepackage[dvipsnames]{xcolor}

\tikzset{snake it/.style={decorate, decoration=snake}}

\usepackage{genyoungtabtikz}
\YFrench
\Yboxdim{0.5cm}

\newcommand{\Ycyan}{\Yfillcolour{cyan!25}}
\newcommand{\Ywhite}{\Yfillcolour{white}}

\newtheorem{thm}{Theorem}[section]
\newtheorem{lem}[thm]{Lemma}

\newtheorem{rem}[thm]{Remark}

\newcommand{\la}{\lambda}
\newcommand{\x}{\mathbf{x}}
\newcommand{\y}{\mathbf{y}}
\newcommand{\A}{\mathcal{A}}

\DeclareMathOperator{\SSYT}{SSYT}

\DeclareMathOperator{\TSPP}{TSPP}
\DeclareMathOperator{\TSPPT}{TSPPT}

\DeclareMathOperator{\sgn}{sgn}
\DeclareMathOperator{\diag}{diag}

\newcommand{\deff}[1]{\emph{\color{cyan!80!blue} #1}} 

\newcommand{\mb}[1]{\mathbf{#1}}
\newcommand{\mc}[1]{\mathcal{#1}}
\newcommand{\wh}[1]{\widehat{#1}}
\newcommand{\wt}[1]{\widetilde{#1}}
\newcommand{\ov}[1]{\overline{#1}}

\newcommand\Square[1]{+(-#1,-#1) rectangle +(#1,#1)}

\newcommand{\RightBoundaryColour}[4]{
\begin{scope}[xshift= 1cm*cos(30)*(#1+#2) , yshift=1cm*#3+1cm*sin(30)*(#2-#1)]
\draw [fill=#4] (0,0) -- ({1*cos(30)},{1*sin(30)}) -- ({1*cos(30)},{-1*sin(30)}) -- (0,-1) -- (0,0);
\end{scope}
}
\newcommand{\LeftBoundaryColour}[4]{
\begin{scope}[xshift= 1cm*cos(30)*(#1+#2) , yshift=1cm*#3+1cm*sin(30)*(#2-#1)]
\draw (0,0)[fill=#4] -- ({-1*cos(30)},{1*sin(30)}) -- ({-1*cos(30)},{-1*sin(30)}) -- (0,-1) -- (0,0);
\end{scope}
}
\newcommand{\TopBoundaryColour}[4]{
\begin{scope}[xshift= 1cm*cos(30)*(#1+#2) , yshift=1cm*#3+1cm*sin(30)*(#2-#1)]
\draw (0,0)[fill=#4] -- ({1*cos(30)},{1*sin(30)}) -- (0,1) -- ({-1*cos(30)},{1*sin(30)}) -- (0,0);
\end{scope}
}

\newcommand{\RightBoundary}[4]{
\ifnum #4 =1
	\RightBoundaryColour{#1}{#2}{#3}{green!80!black}
\else
	\ifnum #4 =2
		\RightBoundaryColour{#1}{#2}{#3}{darkgray}
	\else
		\ifnum #4=3
			\RightBoundaryColour{#1}{#2}{#3}{white}	
		\else
			\ifnum #4=4
				\RightBoundaryMatching{#1}{#2}{#3}	
			\else
				\ifnum #4=5
					\RightBoundaryTriangle{#1}{#2}{#3}
				\else
					\ifnum #4=6
						\RightBoundaryTriangleDotted{#1}{#2}{#3}
					\else
						\ifnum #4=7
							\RightBoundaryColour{#1}{#2}{#3}{white}
							\begin{scope}[xshift= 1cm*cos(30)*(#1+#2) , yshift=1cm*#3+1cm*sin(30)*(#2-#1)]
								\draw[cyan, line width=2pt]  ({1/2*cos(30)},{1/2*sin(30)}) -- ({1/2*cos(30)},{-1+1/2*sin(30)});
							\end{scope}
						\else
							\ifnum #4=8
								\RightBoundaryColour{#1}{#2}{#3}{white}
							\else
								\ifnum #4=9
									\RightBoundaryColour{#1}{#2}{#3}{white}
									\begin{scope}[xshift= 1cm*cos(30)*(#1+#2) , yshift=1cm*#3+1cm*sin(30)*(#2-#1)]
										\draw[red, dashed, line width=2pt]  ({1/2*cos(30)},{1/2*sin(30)}) -- ({1/2*cos(30)},{-1+1/2*sin(30)});
									\end{scope}
								\else
									\RightBoundaryColour{#1}{#2}{#3}{white}
								\fi
							\fi
						\fi
					\fi
				\fi
			\fi
		\fi
	\fi
\fi
}
\newcommand{\LeftBoundary}[4]{
\ifnum #4 =1
	\LeftBoundaryColour{#1}{#2}{#3}{blue}
\else
	\ifnum #4 =2
		\LeftBoundaryColour{#1}{#2}{#3}{lightgray}
	\else
		\ifnum #4=3
			\LeftBoundaryColour{#1}{#2}{#3}{white}	
		\else
			\ifnum #4=4
				\LeftBoundaryMatching{#1}{#2}{#3}	
			\else
				\ifnum #4=5
					\LeftBoundaryTriangle{#1}{#2}{#3}
				\else
					\ifnum #4=6
						\LeftBoundaryTriangleDotted{#1}{#2}{#3}
					\else 
						\ifnum #4=7
							\LeftBoundaryColour{#1}{#2}{#3}{white}	
						\else
							\ifnum #4=8
								\LeftBoundaryColour{#1}{#2}{#3}{white}	
								\begin{scope}[xshift= 1cm*cos(30)*(#1+#2) , yshift=1cm*#3+1cm*sin(30)*(#2-#1)]
									\draw[cyan, line width=2pt]  ({-1/2*cos(30)},{1/2*sin(30)}) -- ({-1/2*cos(30)},{-1+1/2*sin(30)});
								\end{scope}
							\else
								\ifnum #4=9
									\LeftBoundaryColour{#1}{#2}{#3}{white}
								\else
									\LeftBoundaryColour{#1}{#2}{#3}{white}	
									\begin{scope}[xshift= 1cm*cos(30)*(#1+#2) , yshift=1cm*#3+1cm*sin(30)*(#2-#1)]
										\draw[red, dashed, line width=2pt]  ({-1/2*cos(30)},{1/2*sin(30)}) -- ({-1/2*cos(30)},{-1+1/2*sin(30)});
									\end{scope}
								\fi
							\fi
						\fi
					\fi
				\fi
			\fi
		\fi
	\fi
\fi
}
\newcommand{\TopBoundary}[4]{
\ifnum #4 =1
	\TopBoundaryColour{#1}{#2}{#3}{red}
\else
	\ifnum #4 =2
		\TopBoundaryColour{#1}{#2}{#3}{white}
	\else
		\ifnum #4=3
			\TopBoundaryColour{#1}{#2}{#3}{white}	
		\else
			\ifnum #4=4
				\TopBoundaryMatching{#1}{#2}{#3}	
			\else
				\ifnum #4=5
					\TopBoundaryTriangle{#1}{#2}{#3}
				\else
					\ifnum #4=6
						\TopBoundaryTriangleDotted{#1}{#2}{#3}
					\else
						\ifnum #4=7
							\TopBoundaryColour{#1}{#2}{#3}{white}
							\begin{scope}[xshift= 1cm*cos(30)*(#1+#2) , yshift=1cm*#3+1cm*sin(30)*(#2-#1)]
								\draw[cyan, line width=2pt]  ({-1/2*cos(30)},{1-1/2*sin(30)}) -- ({1/2*cos(30)},{1/2*sin(30)});
							\end{scope}
						\else
							\ifnum #4=8
								\TopBoundaryColour{#1}{#2}{#3}{white}
								\begin{scope}[xshift= 1cm*cos(30)*(#1+#2) , yshift=1cm*#3+1cm*sin(30)*(#2-#1)]
									\draw[cyan, line width=2pt]  ({-1/2*cos(30)},{1/2*sin(30)}) -- ({1/2*cos(30)},{1-1/2*sin(30)});
								\end{scope}
							\else
								\ifnum #4=9
									\TopBoundaryColour{#1}{#2}{#3}{white}
									\begin{scope}[xshift= 1cm*cos(30)*(#1+#2) , yshift=1cm*#3+1cm*sin(30)*(#2-#1)]
										\draw[red, dashed, line width=2pt]  ({-1/2*cos(30)},{1-1/2*sin(30)}) -- ({1/2*cos(30)},{1/2*sin(30)});
									\end{scope}
								\else
									\TopBoundaryColour{#1}{#2}{#3}{white}
									\begin{scope}[xshift= 1cm*cos(30)*(#1+#2) , yshift=1cm*#3+1cm*sin(30)*(#2-#1)]
										\draw[red, dashed, line width=2pt]  ({-1/2*cos(30)},{1/2*sin(30)}) -- ({1/2*cos(30)},{1-1/2*sin(30)});
									\end{scope}
								\fi
							\fi
						\fi
					\fi
				\fi
			\fi
		\fi
	\fi
\fi
}

	\newcounter{x}
	\newcounter{y}
	\newcounter{z}
	\newcounter{help}

\newcommand{\PlanePartitionColour}[2]{
	\foreach \m [count=\y] in {#1}{		 								
		\foreach \z [count=\x] in \m{									
			\setcounter{x}{\x}											
			\ifnum \z > 0												
				\TopBoundary{\x}{-\y}{\z}{#2}								
			\fi
			
			\ifnum \x > 1												
				\ifnum \z < \thez										
					\setcounter{help}{\z}
					\addtocounter{help}{1}								
					\foreach \zz in {\thehelp,...,\thez}{					
						\RightBoundary{(\x -1)}{(-\y)}{\zz}{#2}				
					}														
				\fi
			\fi
			\ifnum \y > 1											
				\foreach \mm [count=\yy] in {#1}{					
				\ifnum \yy = \they									
				\foreach \zz [count=\xx] in \mm{					
				\ifnum \xx = \x										
				\ifnum \zz > \z										
					\setcounter{help}{\zz}									
					\addtocounter{help}{-1}	
					\foreach \zzz in {\z,...,\thehelp}{
						\LeftBoundary{(\x -1)}{(1-\yy)}{\zzz}{#2}		
					}
				\fi
				\fi
				}
				\fi
				}
			\fi		
			\setcounter{z}{\z}										
		}										
		\ifnum \thez >0												
			\foreach \z in {1,...,\thez}{
				\RightBoundary{\thex}{-\y}{\z}{#2}						
			}
		\fi
		\setcounter{y}{\y}
	}											

	\foreach \mm [count=\yy] in {#1}{								
	\ifnum \yy = \they
	\foreach \zz [count=\xx] in \mm{
	\ifnum \zz >0
		\foreach \zzz in {1,...,\zz}{
			\LeftBoundary{(\xx -1)}{(1-\yy)}{(\zzz-1)}{#2}			
		}
	\fi
	}
	\fi
	}
	
}

\newcommand{\PlanePartition}[1]{\PlanePartitionColour{#1}{1}}
\newcommand{\PlanePartitionWhite}[1]{\PlanePartitionColour{#1}{3}}

\newcommand{\TilingBox}[4]{
	\foreach \a in {1,...,#2}{
		\foreach \b in {1,...,#3}{
			\RightBoundary{0}{-\a}{\b -1}{#4}
		}
	}
	\foreach \a in {1,...,#1}{
		\foreach \b in {1,...,#3}{
			\LeftBoundary{\a}{0}{\b}{#4}
		}
	}
	\foreach \a in {1,...,#1}{
		\foreach \b in {1,...,#2}{
			\TopBoundary{\a}{-\b}{0}{#4}
		}
	}
}

\title[Determinantal formulae for a symmetric generating function of TSPPs]{Determinantal formulae for a symmetric generating function of totally symmetric plane partitions}

\author{Julia H\"ormayer and Florian Schreier-Aigner}

\begin{document}

\begin{abstract}
Ilse Fischer and the second author introduced in [Algebr. Comb. 7 (2024), no. 5, 1319–1345] a two parameter family of polynomials defined as sums over totally symmetric plane partitions and connected to alternating sign matrices and cyclically symmetric lozenge tilings of a hexagon with a triangular hole.
In this paper we present several determinantal formulae leading to new lattice path models and a novel family of tableaux. The later illustrates that the polynomials of our interest can be thought of as generalisations of the three dual Littlewood identities.
\end{abstract}
\maketitle

\section{Introduction}

Initiated by MacMahon \cite{MacMahon97} at the end of the 19th century, the study of \deff{plane partitions} became popular within the combinatorial community starting in the second half of the last century. In particular, proving the enumeration formulae for the ten symmetry classes of plane partitions became an important research programme which was popularised by Stanley \cite{Stanley86a, Stanley86} and only finished in 2011 by  Koutschan, Kauers, and Zeilberger \cite{KoutschanKauersZeilberger11}. First noted by Bender and Knuth \cite{BenderKnuth72}, the symmetry classes of plane partitions (at least those without cyclical symmetry) were observed to be connected to symmetric polynomials, allowing simple proofs for their enumeration formulae. For more details compare to the overview article by Krattenthaler \cite{Krattenthaler16}.
The symmetry classes of plane partitions involving a cyclical symmetry, on the other hand, turned out to be connected to \deff{alternating sign matrices} (ASMs) introduced by Robbins and Rumsey \cite{RobbinsRumsey86}. Similarly to plane partitions, the early focus was on the enumeration of their symmetry classes; again a very challenging task which was completed only in 2017 by Behrend, Fischer and Konvalinka \cite{BehrendFischerKonvalinka17}. Over the years more numerical connections, i.e., equidistributions, between these two seemingly unrelated combinatorial objects have been found, compare for example to \cite{BehrendDiFrancescoZinnJustin13, FonsecaZinnJustin08, Krattenthaler96}. Nevertheless, their relation is still an unsolved mystery.\medskip

In a recent series of papers by Ilse Fischer and the second author \cite{FischerSchreierAigner23, AignerFischer24}, they introduced a novel weight for \deff{monotone triangles}, combinatorial objects generalising ASMs and in bijection with
semistandard Young tableaux. The corresponding generating functions, nowadays called the \deff{(modified) Robbins polynomials}, are symmetric polynomials generalising Schur polynomials. In the special case corresponding to ASMs, the modified Robbins polynomials reveal a novel and unexpected connection between ASMs and plane partitions: the modified Robbins polynomial can be written in this case as a sum over \deff{totally symmetric plane partitions} (TSPPs) where each TSPP contributes one Schur polynomial. Denote by $\A_{n,k}$ this very generating function of TSPPs, which we call the \deff{symmetric generating function for TSPPs}, where $n$ denotes the size of the involved TSPPs and $k$ is an additional parameter deforming the indexing partition of each of the Schur polynomials in a natural way; the precise definitions are given in Section~\ref{sec:sym gen fct}. As shown in \cite[Theorem 1.2]{AignerFischer24}, the symmetric functions $\A_{n,k}$ yield another intriguing connection to plane partitions, as they enumerate the cyclically symmetric lozenge tilings of cored hexagons with side lengths $n+2k,n,n+2k,n,n+2k,n$ which generalise cyclically symmetric plane partitions.
\medskip

In this paper, we study the symmetric generating function $\A_{n,k}$ for TSPPs from an algebraic combinatorial point of view and derive variants of the Jacobi--Trudi formula, the Nägelsbach--Kostka identity, and the Giambelli identity as presented in our main theorem.

\begin{thm}
\label{thm:main}
The symmetric function $\A_{n+1,k}(\x;r,u,v,w) $ satisfies the following determinantal formulae
\begin{align}
\label{eq:Giambelli 1}
&\det_{1 \leq i,j \leq n} \left(
(-1)^{j-i}v^j\binom{i-1}{j-1}+ r u^i w^{j-i}s_{(i-1|j+k-1)}(\x)
 \right) \\
\label{eq:dual JT} =& \det_{1 \leq i,j \leq n} \left(
 r u^i e_{(n+k+i-j)}(\x) +(-1)^{i-1}v^{i}\sum_{m}\binom{m-1}{m-i}e_{(n+1-j-m)}(\x) \left(\frac{w}{u}\right)^{m-i}
 \right)\\
\label{eq:JT} =& \det_{1 \leq i,j \leq n+k}\left(
\begin{matrix}
   v^{n-i+1} h_{i-j}(\x) +(-1)^{k+n-i}ru^{n-i+1}\sum\limits_{m=0}^{n-i}\binom{n-i}{m-i} h_{k+1+2n-m-j}(\x)  \left(\frac{w}{u}\right)^{m-i} \: 1\leq i \leq n \\
  \qquad\qquad\qquad\qquad\qquad  h_{i-j}(\x) \qquad\qquad\qquad\qquad\qquad\qquad\quad n+1\leq i \leq n+k
\end{matrix}
\right) \\
\label{eq:Giambelli 2} =&  \det_{1 \leq i,j \leq n} \left(
 \delta_{i,j}v^{i}+r u^{j} \sum_{m} \binom{i-1}{i-m}
\left(\frac{w}{u}\right)^{i-m} s_{(m-1|j+k-1)}(\x)
 \right),
\end{align}
\end{thm}

In particular we provide lattice path interpretations for the last three determinants using the Lindström--Gessel-Viennot lemma \cite{GesselViennot85, Lindstroem73} and use the method of \emph{dualisation} due to Gessel and Viennot \cite{GesselViennot85} for a combinatorial argument showing that these determinants are equal. Moreover we present a tableaux theoretic interpretation for these determinants which we call the \deff{TSPP tableaux}. Using these TSPP tableaux it is immediate that the symmetric functions $\A_{n,k}$ are connected to the three dual Littlewood identities, as shown in \eqref{eq:n,1 case}, \eqref{eq:n,0 case} and \eqref{eq:n,-1 case}. We believe that these new tableaux will be a useful tool for working with $\A_{n,k}$.\medskip

We start by introducing all necessary symmetric polynomial terminology in Section~\ref{sec:preliminaries}. In Section~\ref{sec:sym gen fct}, we define the symmetric generating function $\A_{n,k}$. Section~\ref{sec:proof} contains the proof of our main theorem. Thereafter, in Section~\ref{sec:TSPP tab}, we introduce TSPP tableaux and show that they yield another combinatorial interpretation for $\A_{n,k}$.

\section{Preliminaries}
\label{sec:preliminaries}

A \deff{partition} $\la$ is a weakly decreasing sequence of positive integers $\la=(\la_1,\ldots,\la_\ell)$. We say that $\la$ is a partition of $|\la|:=\la_1+\cdots+\la_\ell$, denoted by $\la \vdash |\la|$, and call $\ell$ its \deff{length}, also denoted by $\ell(\la)$. 
 The \deff{Young diagram} of a partition $\la$ is a left-aligned collection of boxes such that the $i$-th row from the bottom has $\la_i$ boxes (French convention). For simplicity, we identify partitions with their corresponding Young diagrams. The \deff{conjugate} $\la^\prime=(\la_1^\prime,\la_2^\prime,\ldots)$ of $\la$ is obtained by reflecting $\la$ along the $x=y$ diagonal, see Figure~\ref{fig:SSYT} for an example.
 The \deff{Frobenius notation} of $\la$ is defined as $(\la_1-1, \ldots, \la_d-d| \la_1^\prime-1,\ldots, \la_d^\prime-d)$ where $d$ is the size of the \deff{Durfee square}, i.e., the maximal integer $d$ with $\la_d \geq d$. 
A \deff{semistandard Young tableaux} (SSYT) of shape $\la$ is a filling of $\la$ with positive integers such that rows are weakly decreasing from left to right and columns are strictly decreasing from bottom to top. For an SSYT $T$ we define its weight by
\[
\x^T=x_1^{\#\text{ of 1's in } T}x_2^{\# \text{ of 2's in } T}\cdots x_n^{\# \text{ of n's in } T}.
\]
Denote by $\SSYT_\la$ the set of SSYTs of shape $\la$. For an infinite family of variables $\x=(x_1,x_2,\ldots)$, we define the \deff{Schur function} $s_\la(\x)$ as the multivariate generating function for SSYTs of shape~$\la$
\[
s_\la(\x):=\sum_{T \in \SSYT_\la}\x^T.
\]
Two important special cases of the Schur function are the \deff{complete homogeneous symmetric function} $h_k(\x) = s_{(k)}(\x)$ and the \deff{elementary symmetric function} $e_k(\x) = s_{(1^k)}(\x)$, where $(1^k)$ denotes the partition consisting of $k$ times the entry $1$.
\begin{figure}
\begin{center}
\begin{tikzpicture}
\tyoung(0cm,0cm,1224,23,4)
\tyng(3.5cm,0cm, 3,2,1,1)
\end{tikzpicture}
\caption{\label{fig:SSYT} A SSYT of shape $(4,2,1)$ and weight $x_1 x_2^{3}x_3 x_4^2$ (left) and the conjugate  of the partition $(4,2,1)$ (right).}
\end{center} 
\end{figure}
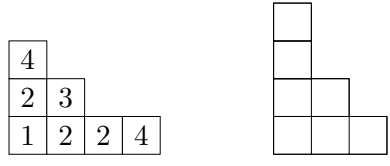
\bigskip

There are three determinantal formulae for Schur functions which are of relevance for this paper: the \deff{Jacobi--Trudi formula}
\[
s_\la(\x) = \det_{1 \leq i,j \leq \ell(\la)} \Big(h_{\la_i-i}(\x) \Big),
\]
the \deff{Nägelsbach--Kostka identity}, or \deff{dual Jacobi--Trudi formula}
\[
s_\la(\x) = \det_{1 \leq i,j \leq n} \Big( e_{\la_i^\prime-i}(\x) \Big),
\]
and the \deff{Giambelli identity}
\[
s_\la(\x) = \det_{1 \leq i,j \leq d} \Big( s_{(a_i|b_j)}(\x) \Big).
\]
All of these formulae can be proven combinatorially by the celebrated Lindström--Gessel--Viennot lemma \cite{GesselViennot85, Lindstroem73}, introduced next. Let $G$ be an acyclic graph with an edge weight $\omega$. We define two paths in $G$ to be \deff{non-intersecting} if they do not have a vertex in common and a family of paths in $G$ to be non-intersecting if all pairs of paths within this family are non-intersecting. We define the \deff{weight of a path} $\omega(P)$ as the product of its edge weights
\[
\omega(P):=\prod_{e \in P} \omega(P),
\]
and the weight of a family of paths as the product of the weights of the individual paths.

\begin{lem}[\cite{GesselViennot85, Lindstroem73}]
\label{lem:LGV}
Given an acyclic graph $G$ with an edge weight, a set $\mb{S}=\{S_1,\ldots,S_n\}$ of starting points and $\mb{E}=\{E_1,\ldots,E_n\}$ of end points, we say that a family of paths from $\mb{S}$ to $\mb{E}$ has sign $\pi$ if the path starting at $S_i$ ends at $E_{\pi(i)}$ for all $i$. Then the signed and weighted enumeration of all non-intersecting families of paths from $\mb{S}$ to $\mb{E}$ is given by the determinant
\[
\det_{1 \leq i,j \leq n}\Big( \mc{P}(i,j) \Big),
\]
where $\mc{P}(i,j)$ is the weighted generating function of all paths from $S_i$ to $E_j$.
\end{lem}

\section{The symmetric generating function for totally symmetric plane partitions}
\label{sec:sym gen fct}
A \deff{plane partition} $\pi=(\pi_{i,j})_{1 \leq i \leq a, 1 \leq j \leq b}$ inside an $(a,b,c)$-box is an array of integers in $[0,c]$ such that rows and columns are weakly decreasing, i.e., $\pi_{i,j}  \geq \pi_{i+1,j}$ and $\pi_{i,j} \geq \pi_{i,j+1}$. 
They can be visualised graphically as stacks of cubes by replacing the entry $a_{i,j}$ by $a_{i,j}$ many unit cubes, see Figure~\ref{fig:PPs} for an example. As observed by David and Tomei \cite{DavidTomei89}, by considering a two dimensional picture of such a configuration, one obtains bijectively a lozenge tiling of a regular hexagon with side lengths $(a,b,c,a,b,c)$. For simplicity we identify a plane partition with its corresponding lozenge tiling.
\begin{figure}[h]
\begin{center}
 \begin{tikzpicture}
 \begin{scope}[scale=0.65]
  \node at (1,3) {4};
\node at (2,3) {3};
\node at (3,3) {3};
\node at (4,3) {1};
\node at (1,2) {4};
\node at (2,2) {2};
\node at (3,2) {1};
\node at (4,2) {0};
\node at (1,1) {2};
\node at (2,1) {0};
\node at (3,1) {0};
\node at (4,1) {0};
 \end{scope}
 \begin{scope}[scale=0.4, xshift=12cm, yshift=2cm]
 \PlanePartition{{4,3,3,1},{4,2,1,0},{2,0,0}}
 \end{scope}
  \begin{scope}[scale=0.4, xshift=21cm, yshift=2cm]
  \TilingBox{4}{3}{4}{3}
 \PlanePartitionWhite{{4,3,3,1},{4,2,1,0},{2,0,0}}
 \end{scope}
 \end{tikzpicture}
\caption{\label{fig:PPs} A plane partition inside a $(3,4,4)$-box (left), its graphical representation as stacks of cubes (middle) and its corresponding lozenge tiling of a hexagon (right).}
\end{center}
\end{figure}
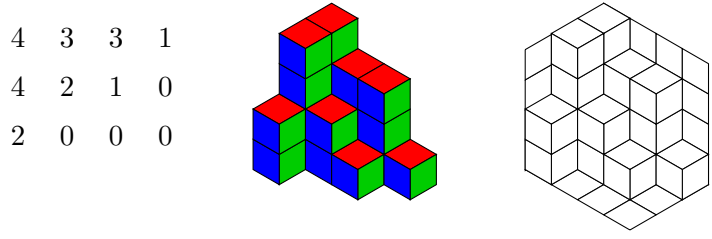
 For this paper, the following symmetry classes are of relevance.  We call a lozenge tiling (and also its corresponding plane partition)
\begin{itemize}
\item \deff{symmetric} if it is invariant under vertical reflection and
\item \deff{cyclically symmetric} if it is invariant under rotation by $120$ degrees.
\end{itemize}
A plane partition is called \deff{totally symmetric}\footnote{Note that the corresponding lozenge tiling of a totally symmetric plane partition is not symmetric with respect to the full dihedral group of the hexagon (these are called totally symmetric self-complementary) as this notion was introduced before the relation to lozenge tilings was found.} if it is both symmetric and cyclically symmetric. Denote by $\TSPP_n$ the set of \deff{totally symmetric plane partitions} (TSPPs).
For a totally symmetric plane partition $T=(T_{i,j})_{1\leq i,j \leq n} \in \TSPP_n$, we define its diagonal $\diag(T)=(T_{i,i})_{1 \leq i \leq n}^\prime$ and the partition $\pi_k(T)=(a_1,\ldots,a_d|b_1+k,\ldots,b_d+k)$ where $(a_1,\ldots,a_d|b_1,\ldots,b_d)$ is the Frobenius notation of $\diag(T)$ as well as the weight 
\[
\omega(T):=r^d u^{\sum_{i=1}^d (a_i+1)}  v^{\binom{n}{2}-\sum_{i=1}^d (b_i+1)} w^{\sum_{i=1}^d (b_i-a_i)}.
\]
Analogously to the term \emph{(exponential) generating function}, we call the family $\A_{n,k}$ of symmetric functions in the variables $\x=(x_1,x_{2},\ldots)$ 
\begin{equation}
\label{eq:sym gen fct}
\A_{n,k}(\x;r,u,v,w)= \sum_{T \in \TSPP_{n-1}} \omega(T) s_{\pi_k(T)}(\x),
\end{equation}
the \deff{symmetric generating function for TSPPs}.
\bigskip

It was shown by the second author and Ilse Fischer \cite[Theorem 1.1]{AignerFischer24} that by restricting to the finite family of variables $(x_1,\ldots,x_{n+k-1})$, the symmetric polynomial $\A_{n,1}(\x;1,u,v,w)$ is equal to a novel multivariate generating function for alternating sign matrices introduced in \cite{FischerSchreierAigner23} and furthermore  \cite[Theorem 1.2]{AignerFischer24} that $\A_{n,k}$  yields refined enumeration formulae for cyclically symmetric lozenge tilings of a cored hexagon which generalise cyclically symmetric plane partitions.

\section{Proof of Theorem~\ref{thm:main}}
\label{sec:proof}
Our proof strategy is as follows. Using the Lindström--Gessel--Viennot lemma (Lemma~\ref{lem:LGV}), we present in Section~\ref{sec:epaths} a combinatorial interpretation of \eqref{eq:dual JT}. By using the method of \emph{dualising} introduced by Gessel and Viennot in \cite{GesselViennot85}, we obtain in Section~\ref{sec:hpaths} the combinatorial model for \eqref{eq:JT} and by using a mixed dualisation for \eqref{eq:Giambelli 2}. Again, the corresponding determinants are obtained by applying the Lindström--Gessel--Viennot lemma. Finally we prove algebraically the equality of the determinants \eqref{eq:Giambelli 1} and \eqref{eq:Giambelli 2} in Section~\ref{sec:mixedpaths}. This finishes the proof as it is shown in \cite[Lemma 6.2]{AignerFischer24} that \eqref{eq:Giambelli 1} is equal to $\A_{n+1,k}(\x;r,u,v,w)$.

\subsection{A combinatorial interpretation for \eqref{eq:dual JT}}
\label{sec:epaths}

For $1 \leq i,j \leq n$, define the points
\[
S_i=(1-i,i-1), \quad S_i^\prime = (k+1,-k-2i+1)\quad \textit{and} \quad E_j=(k+n+1-j,\infty).
\]
We regard non-intersecting lattice paths where the $i$-th path starts either at $S_i$ or $S_i^\prime$ and ends at an $E_j$. For a point $(a,b)$, there are two step sets depending on its position.
\begin{itemize}
\item For $a+b \geq 0$, we are allowed to take the steps $(1,0)$ and $(0,1)$ where the first is weighted by $x_{a+b+1}$ and the second by $1$.
\item If $a>k$ and $a+b \leq 0$, the step set consists of $(1,-1)$ and $(1,1)$ where the first step is weighted by $\frac{w}{u}$ and the second again by $1$.
\item On the intersection of both regions, we are allowed to take any of these steps.
\end{itemize}
For simplicity, we call such a family of paths a \deff{family of e-paths of size $(n,k)$}; see Figure~\ref{fig:e-path} for an example. The weight of a family of e-paths is the product over all step weights times the product over $r u^i$ if the $i$-th path starts at $S_i$ and $v^i$ if the $i$-th path starts at $S_i^\prime$ for $1 \leq i \leq n$.
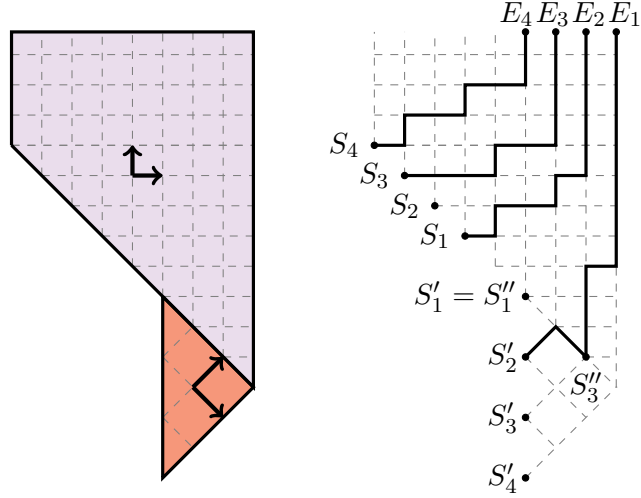
\begin{figure}[h]
\begin{tikzpicture}[scale=.4]
\foreach \i in {-5,...,3}{
		\draw[dashed ,gray] (-1*\i,\i) -- (5,\i);
		\draw[dashed ,gray] (-1*\i,\i) -- (-1*\i,6.75);
}
\foreach \i in {3,...,6}{
		\draw[dashed ,gray] (-3,\i) -- (5,\i);
}
\foreach \i in {1,...,4}{
		\draw[dashed ,gray] (2,-2*\i) -- (6-\i,-4-\i);
		\draw[dashed ,gray] (2,-2*\i) -- (1+\i,-1-\i);
}

\filldraw (0,0) circle (3pt);
\filldraw (-1,1) circle (3pt);
\filldraw (-2,2) circle (3pt);
\filldraw (-3,3) circle (3pt);
\node at (-1,0) {$S_1$};
\node at (-2,1) {$S_2$};
\node at (-3,2) {$S_3$};
\node at (-4,3) {$S_4$};

\filldraw (2,-2) circle (3pt);
\filldraw (2,-4) circle (3pt);
\filldraw (2,-6) circle (3pt);
\filldraw (2,-8) circle (3pt);
\node at (0,-2) {$S_1^\prime = S_1^{\prime\prime}$};
\node at (1.25,-4) {$S_2^\prime$};
\node at (1.25,-6) {$S_3^\prime$};
\node at (1.25,-8) {$S_4^\prime$};

\filldraw (5,6.75) circle (3pt);
\filldraw (4,6.75) circle (3pt);
\filldraw (3,6.75) circle (3pt);
\filldraw (2,6.75) circle (3pt);
\node at (5.3,7.35) {$E_1$};
\node at (4.1,7.35) {$E_2$};
\node at (2.9,7.35) {$E_3$};
\node at (1.7,7.35) {$E_4$};

\filldraw (4,-4) circle (3pt);
\node at (4,-5) {$S_3^{\prime\prime}$};

\draw[very thick] (2,-4) -- (3,-3) -- (4,-4) -- (4,-1) -- (5,-1) -- (5,6.75);
\draw[very thick] (0,0) -- (1,0) -- (1,1) -- (3,1) -- (3,2) -- (4,2) -- (4,6.75);
\draw[very thick] (-2,2) -- (1,2) -- (1,3) -- (2,3) -- (3,3) -- (3,6.75);
\draw[very thick] (-3,3) -- (-2,3) -- (-2,4) -- (0,4) -- (0,5) -- (2,5) -- (2,6.75);

\begin{scope}[xshift=-12cm]
\filldraw[white!50!Red] (2,-2) -- (5,-5) -- (2,-8) -- (2,-2);
\filldraw[white!80!Orchid] (-3,3) -- (5,-5) -- (5,6.75) -- (-3,6.75) -- (-3,3);

\foreach \i in {-5,...,3}{
		\draw[dashed ,gray] (-1*\i,\i) -- (5,\i);
		\draw[dashed ,gray] (-1*\i,\i) -- (-1*\i,6.75);
}
\foreach \i in {3,...,6}{
		\draw[dashed ,gray] (-3,\i) -- (5,\i);
}
\foreach \i in {1,...,4}{
		\draw[dashed ,gray] (2,-2*\i) -- (6-\i,-4-\i);
		\draw[dashed ,gray] (2,-2*\i) -- (1+\i,-1-\i);
}

\draw[very thick] (-3,3) -- (5,-5) -- (5,6.75) -- (-3,6.75) -- (-3,3);
\draw[very thick] (2,-2) -- (5,-5) -- (2,-8) -- (2,-2);

\draw[->, ultra thick] (1,2) -- (2,2);
\draw[->, ultra thick] (1,2) -- (1,3);

\draw[->, ultra thick] (3,-5) -- (4,-4);
\draw[->, ultra thick] (3,-5) -- (4,-6);
\end{scope}
\end{tikzpicture}
\caption{\label{fig:e-path} On the left: the two regions defining the step set coloured in purple and red respectively together with the allowed steps represented as arrows. On the right: a family of $e$-paths of rank $(4,1)$.}
\end{figure}
By the Lindstr\"om--Gessel--Viennot lemma and Remark~\ref{rem:LGV variant} (see below), it suffices to calculate the generating function of the $i$-th e-path starting at $S_i$ or $S_j$ respectively and ending at $E_j$.
 The generating function of all paths starting at $S_i$ and ending at $E_j$ is $r u^i e_{(n-i+k)}$ since there are $n-i+k$ horizontal steps with non-repeating weights. Given a path starting at $S_i^\prime$, we split it into two parts. The first part consists only of $(1,1)$ and $(1,-1)$ steps and ends at $S_m^{\prime\prime}=(m+k,-m-k)$ and the second part is from $S_m^{\prime\prime}$ to $E_j$ using $(1,0)$ and $(0,1)$ steps. For a fixed $m$, the generating function for such a first part from $S_i^\prime$ to $S_m^{\prime\prime}$ is $\binom{m-1}{m-i}\left(\frac{w}{u}\right)^{m-i}$ since we have to take $m-1$ steps and exactly $m-i$ of them are $(1,-1)$. The generating function for paths from $S_m^{\prime\prime}$ to $E_j$ is $e_{(n+1-j-m)}(\x)$.
Hence the generating function of all paths from $S_i^\prime$ to $E_j$ is equal to
\[
v^{i}\sum_{m}\binom{m-1}{m-i}e_{(n+1-j-m)}(\x) \left(\frac{w}{u}\right)^{m-i}.
\]
By applying Remark~\ref{rem:LGV variant}, we obtain the assertion.

\begin{rem}
\label{rem:LGV variant}
\begin{enumerate}
\item We still owe an explanation as to why we are allowed to use the Lindström-Gessel-Viennot lemma in our setting. Let the graph and step weights be as above and let the weight of a family of paths be the product over all step weights times the product over $a_i$ if the $i$-th path starts at $S_i$ and $b_i$ if the $i$-th path starts at $S_i^\prime$ for $1\leq i \leq n$. Since the lemma interprets the determinant as a signed enumeration of lattice paths, we need to deal with the fact that the sign of the non-crossing families changes every time a starting point $S_i^\prime$ instead of $S_i$ is chosen. For this, fix a subset $I \subseteq [1,n]$ and a family of paths where the $i$-th path starts at $S_i^\prime$ if $i\in I$ and at $S_i$ otherwise. For $I=\{i_1 < \cdots < i_l\}$ and $[1,n]\setminus I= \{ j_1 < \cdots < j_k\}$, denote by $\sigma_I$ the permutation $i_l,\ldots,i_1,j_1,\ldots,j_k$ in one line notation and let $\pi$ be the permutation such that the $i$-th path ends at $E_{\pi(i)}$. For $I = \emptyset$, we are back in the classical setting. 
We can reduce the case of a general $I$ to the previous case by relabelling the starting indices. More precisely, we obtain the labels $1,\ldots,n$ from right to left for the starting vertices by applying the permutation $\sigma_I^{-1}$ to the starting indices.  
This yields a total sign of $\sgn(\pi) \sgn(\sigma_I^{-1})$ which differs from the classical case by a factor of $\sgn(\sigma_I^{-1})= \prod_{i \in I}(-1)^{i-1}$. This sign can be envisioned as each chosen $S_i^\prime$ having to switch places with $i-1$ chosen starting points (for each $t<i$, either $S_t$ or $S^\prime_t$ - they are both on the way) in order to end up in the spot of $S_i$. 
 We claim that the weighted enumeration of all non-intersecting families of such paths is equal to the determinant
\begin{equation}
\label{eq:LGV variation}
\det_{1 \leq i,j \leq n}\Big( a_i \mc{P}(i,j) + (-1)^{i-1} b_i \mc{P}^\prime(i,j) \Big),
\end{equation}
where $\mc{P}(i,j)$, respectively $\mc{P}^\prime(i,j)$, is the weighted generating function of all paths from $S_i$, respectively $S_i^\prime$, to $E_j$.
Indeed, for a given subset $I \subseteq [1,n]$, the weighted and signed enumeration of all non-intersecting families, where the $i$-th path starts at $S_i^\prime$ if $i \in I$ and at $S_i$ otherwise, is given by
\begin{multline*}
\sgn(\sigma_I^{-1})\prod_{i \in [1,n]\setminus I}a_i \prod_{i \in I}b_i 
 \det_{1 \leq i,j \leq n}  \left( 
\begin{cases}
\mc{P}(i,j) \quad & i \notin I,\\
\mc{P}^\prime(i,j) & i \in I,
\end{cases}
\right) \\
= 
\det_{1 \leq i,j \leq n} \left( 
\begin{cases}
a_i \mc{P}(i,j) \quad & i \notin I,\\
(-1)^{i-1}b_i \mc{P}^\prime(i,j) & i \in I,
\end{cases}
\right)
.
\end{multline*}
Using the multilinearity of the determinant, we obtain the anticipated result by summing over all subsets $I \subseteq [1,n]$.

\item If instead we find ourselves in the setting of Subsection~\ref{sec:dualising}, i.e. the labels are as in the right picture in Figure~\ref{fig:h-path}, then we want $\sigma_I$ to be the permutation $j_1,\ldots,j_k, i_l,\ldots,i_1$ in one line notation with $\sgn(\sigma_I)=\prod_{i \in I}(-1)^{n-i}$. Hence, we exchange $(-1)^{i-1}$ by $(-1)^{n-i}$ in \eqref{eq:LGV variation}.
\end{enumerate}

\end{rem}

\subsection{Dualising the paths}\label{sec:dualising}
\label{sec:hpaths}
We will first define a \deff{family of h-paths of size $(n,k)$}, then we will show, using a concept called \deff{dualisation}, that they are in bijection with the \deff{family of e-paths of size $(n,k)$} defined above and finally, we will compute its generating function to obtain formula~\eqref{eq:JT}.
Thus, for $1\leq i,j\leq n+k$, we define the following points
\begin{align*}
\wt{E}_j = (j-n,\infty) \quad \textit{and} \quad
\wt{S}_i = (i-n,-i+n),
\end{align*}
as well as for $1\leq i,j\leq n$, we define
\begin{align*}
\wt{S}_i^\prime = (k+1,-2n-k+2i-1) \quad \textit{and} \quad
\wt{S}_{m}^{\prime\prime} = (k+n+1-m,-k-n-1+m).
\end{align*}
We look at non-intersecting lattice paths with the $i$-th path starting either at $\wt{S}_i$ or $\wt{S}_i^\prime$ and ending at an $\wt{E}_j$. For a point $(a,b)$, there are two step sets depending on its position.
\begin{itemize}
\item For $a+b \geq 0$, we are allowed to take the steps $(-1,1)$ and $(0,1)$ where the first is weighted by $x_{a+b+1}$ and the second by $1$.
\item If $a>k$ and $a+b \leq 0$, the step set consists of $(0,2)$ and $(1,1)$ where the first step is weighted by $\frac{w}{u}$ and the second again by $1$.
\item On the intersection of both regions we are allowed to take any of these steps.
\end{itemize}
The weight of a family of h-paths is the product over all step weights times the product over $r u^{n+1-i}$ if the $i$-th path starts at $\wt{S}_i^\prime$ and $v^{n+1-i}$ if the $i$-th path starts at $\wt{S}_i$ for $1 \leq i \leq n$.
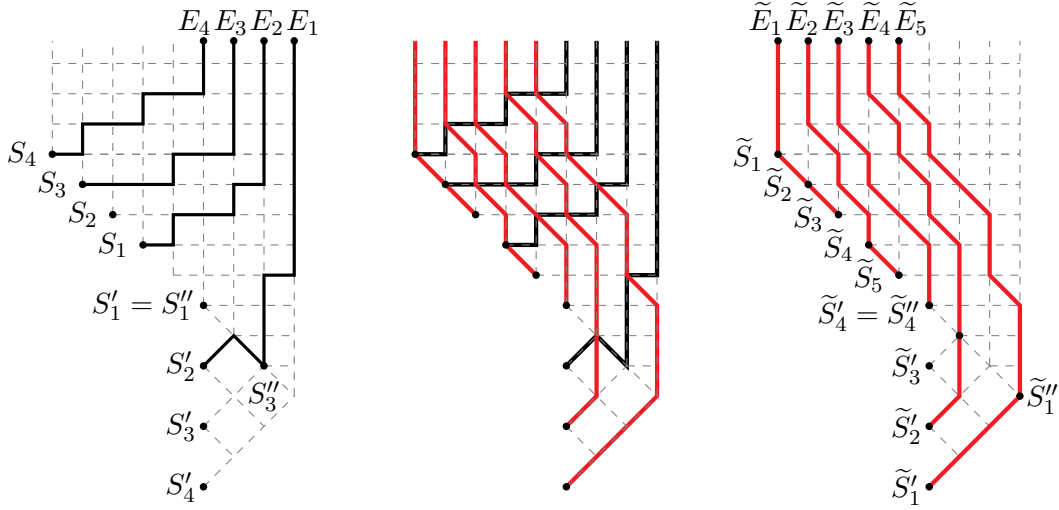
\begin{figure}[h]
\begin{tikzpicture}[scale=.4]

\begin{scope}[xshift=-12cm]
\foreach \i in {-5,...,3}{
		\draw[dashed ,gray] (-1*\i,\i) -- (5,\i);
		\draw[dashed ,gray] (-1*\i,\i) -- (-1*\i,6.75);
}
\foreach \i in {3,...,6}{
		\draw[dashed ,gray] (-3,\i) -- (5,\i);
}
\foreach \i in {1,...,4}{
		\draw[dashed ,gray] (2,-2*\i) -- (6-\i,-4-\i);
		\draw[dashed ,gray] (2,-2*\i) -- (1+\i,-1-\i);
}

\filldraw (0,0) circle (3pt);
\filldraw (-1,1) circle (3pt);
\filldraw (-2,2) circle (3pt);
\filldraw (-3,3) circle (3pt);
\node at (-1,0) {$S_1$};
\node at (-2,1) {$S_2$};
\node at (-3,2) {$S_3$};
\node at (-4,3) {$S_4$};

\filldraw (2,-2) circle (3pt);
\filldraw (2,-4) circle (3pt);
\filldraw (2,-6) circle (3pt);
\filldraw (2,-8) circle (3pt);
\node at (0,-2) {$S_1^\prime = S_1^{\prime\prime}$};
\node at (1.25,-4) {$S_2^\prime$};
\node at (1.25,-6) {$S_3^\prime$};
\node at (1.25,-8) {$S_4^\prime$};

\filldraw (5,6.75) circle (3pt);
\filldraw (4,6.75) circle (3pt);
\filldraw (3,6.75) circle (3pt);
\filldraw (2,6.75) circle (3pt);
\node at (5.3,7.35) {$E_1$};
\node at (4.1,7.35) {$E_2$};
\node at (2.9,7.35) {$E_3$};
\node at (1.7,7.35) {$E_4$};

\filldraw (4,-4) circle (3pt);
\node at (4,-5) {$S_3^{\prime\prime}$};

\draw[very thick] (2,-4) -- (3,-3) -- (4,-4) -- (4,-1) -- (5,-1) -- (5,6.75);
\draw[very thick] (0,0) -- (1,0) -- (1,1) -- (3,1) -- (3,2) -- (4,2) -- (4,6.75);
\draw[very thick] (-2,2) -- (1,2) -- (1,3) -- (2,3) -- (3,3) -- (3,6.75);
\draw[very thick] (-3,3) -- (-2,3) -- (-2,4) -- (0,4) -- (0,5) -- (2,5) -- (2,6.75);
\end{scope}

\draw[ultra thick] (2,-4) -- (3,-3) -- (4,-4) -- (4,-1) -- (5,-1) -- (5,6.75);
\draw[ultra thick] (0,0) -- (1,0) -- (1,1) -- (3,1) -- (3,2) -- (4,2) -- (4,6.75);
\draw[ultra thick] (-2,2) -- (1,2) -- (1,3) -- (2,3) -- (3,3) -- (3,6.75);
\draw[ultra thick] (-3,3) -- (-2,3) -- (-2,4) -- (0,4) -- (0,5) -- (2,5) -- (2,6.75);
\draw[ultra thick, Red] (-1,1) -- (-3,3) -- (-3,6.75);
\draw[ultra thick, Red] (1,-1) -- (0,0) -- (0,1) -- (-1,2) -- (-1,3) -- (-2,4) -- (-2,6.75);
\draw[ultra thick, Red] (2,-2) -- (2,0) -- (0,2) -- (0,3) -- (-1,4) -- (-1,6.75);
\draw[ultra thick, Red] (2,-6) -- (3,-5) -- (3,-3) -- (3,0) -- (2,1) -- (2,2) -- (1,3) -- (1,4) -- (0,5) -- (0,6.75);
\draw[ultra thick, Red] (2,-8) -- (5,-5) -- (5,-2) -- (4,-1) -- (4,1) -- (2,3) -- (2,4) -- (1,5) -- (1,6.75);

\foreach \i in {-5,...,3}{
		\draw[dashed ,gray] (-1*\i,\i) -- (5,\i);
		\draw[dashed ,gray] (-1*\i,\i) -- (-1*\i,6.75);
}
\foreach \i in {3,...,6}{
		\draw[dashed ,gray] (-3,\i) -- (5,\i);
}
\foreach \i in {1,...,4}{
		\draw[dashed ,gray] (2,-2*\i) -- (6-\i,-4-\i);
		\draw[dashed ,gray] (2,-2*\i) -- (1+\i,-1-\i);
}

\filldraw (1,-1) circle (3pt)node[below left] {};
\filldraw (0,0) circle (3pt)node[below left] {};
\filldraw (-1,1) circle (3pt)node[below left] {};
\filldraw (-2,2) circle (3pt)node[below left] {};
\filldraw (-3,3) circle (3pt)node[below left] {};

\filldraw (2,-2) circle (3pt)node[below left] {};
\filldraw (2,-4) circle (3pt)node[below left] {};
\filldraw (2,-6) circle (3pt)node[below left] {};
\filldraw (2,-8) circle (3pt)node[below left] {};

\begin{scope}[xshift=12cm]

    \foreach \i in {-5,...,3}{
		\draw[dashed ,gray] (-1*\i,\i) -- (5,\i);
		\draw[dashed ,gray] (-1*\i,\i) -- (-1*\i,6.75);
}
\foreach \i in {3,...,6}{
		\draw[dashed ,gray] (-3,\i) -- (5,\i);
}
\foreach \i in {1,...,4}{
		\draw[dashed ,gray] (2,-2*\i) -- (6-\i,-4-\i);
		\draw[dashed ,gray] (2,-2*\i) -- (1+\i,-1-\i);
}

\draw[ultra thick, Red] (-1,1) -- (-3,3) -- (-3,6.75);
\draw[ultra thick, Red] (1,-1) -- (0,0) -- (0,1) -- (-1,2) -- (-1,3) -- (-2,4) -- (-2,6.75);
\draw[ultra thick, Red] (2,-2) -- (2,0) -- (0,2) -- (0,3) -- (-1,4) -- (-1,6.75);
\draw[ultra thick, Red] (2,-6) -- (3,-5) -- (3,-3) -- (3,0) -- (2,1) -- (2,2) -- (1,3) -- (1,4) -- (0,5) -- (0,6.75);
\draw[ultra thick, Red] (2,-8) -- (5,-5) -- (5,-2) -- (4,-1) -- (4,1) -- (2,3) -- (2,4) -- (1,5) -- (1,6.75);

\filldraw (1,-1) circle (3pt);
\filldraw (0,0) circle (3pt);
\filldraw (-1,1) circle (3pt);
\filldraw (-2,2) circle (3pt);
\filldraw (-3,3) circle (3pt);
\node at (0,-1) {$\wt{S}_5$};
\node at (-1,0) {$\wt{S}_4$};
\node at (-2,1) {$\wt{S}_3$};
\node at (-3,2) {$\wt{S}_2$};
\node at (-4,3) {$\wt{S}_1$};

\filldraw (2,-2) circle (3pt);
\filldraw (2,-4) circle (3pt);
\filldraw (2,-6) circle (3pt);
\filldraw (2,-8) circle (3pt);
\node at (0,-2.25) {$\wt{S}_4^\prime=\wt{S}_4^{\prime\prime}$};
\node at (1.25,-4) {$\wt{S}_3^\prime$};
\node at (1.25,-6) {$\wt{S}_2^\prime$};
\node at (1.25,-8) {$\wt{S}_1^\prime$};

\filldraw (-3,6.75) circle (3pt);
\filldraw (-2,6.75) circle (3pt);
\filldraw (-1,6.75) circle (3pt);
\filldraw (0,6.75) circle (3pt);
\filldraw (1,6.75) circle (3pt);
\node at (-3.4,7.5) {$\wt{E}_1$};
\node at (-2.2,7.5) {$\wt{E}_2$};
\node at (-1,7.5) {$\wt{E}_3$};
\node at (0.2,7.5) {$\wt{E}_4$};
\node at (1.4,7.5) {$\wt{E}_5$};

\filldraw (2,-2) circle (3pt);
\filldraw (3,-3) circle (3pt);
\filldraw (5,-5) circle (3pt);
\node at (5.75,-5) {$\wt{S}_1^{\prime\prime}$};
\end{scope}
\end{tikzpicture}
\caption{\label{fig:h-path} On the left: a family of $e$-paths of rank $(4,1)$. On the right: its dualised family of $h$-paths. In the middle: Both above each other.}
\end{figure}
Notice that for $i\leq n$, we have the same possible starting points as for the e-paths but they are indexed the other way around (compare the left and the right pictures in Figure~\ref{fig:h-path}) and for $i\geq n+1$, we get additional starting points that now leave no choice.\\

 Next, we will give a weight-preserving bijection between all such families of e-paths of size $(n,k)$ and all families of h-paths of size $(n,k)$ by \deff{dualisation}. For this, take a family of e-paths, $\mathcal{P}$, and map it to the family of h-paths, $\wt{\mathcal{P}}$, which is defined by the following.
\begin{itemize}
    \item The starting points are all those $\wt{S}_i$ and $\wt{S}_i^\prime$ that are not already starting points in $\mathcal{P}$. In practice, this means that $\wt{S}_i^\prime$ with $i\geq n+1$ are always starting points and for $i\leq n$, we choose $\wt{S}_i^\prime$ if $S_{n+1-i}$ is a starting point in $\mathcal{P}$ and $\wt{S}_i$ if $S_{n+1-i}^\prime$ is a starting point in $\mathcal{P}$.
    \item For $a+b\geq 0$, whenever there is a step $(1,0)$ at $(a,b)$ in $\mathcal{P}$, we take a step $(-1,1)$ at $(a+1,b-1)$ in $\wt{\mathcal{P}}$. Visually, this can be seen as "tilting down" the $(1,0)$ step by $45^\circ$ clockwise. All other steps are $(0,1)$. For  $a+b\leq 0$, we do the analogous thing by converting $(1,-1)$ steps into $(0,1)$ steps and letting the rest be $(1,1)$ steps.
    
\end{itemize}
To see that this indeed describes a weight-preserving bijection, we best look at the overlap of $\mathcal{P}$ with $\wt{\mathcal{P}}$ as in Figure ~\ref{fig:h-path}. First of all, we see immediately that this operation preserves weights by definition. The mapping is also clearly injective. By construction, the dual exactly covers all the points not covered before, with the direction of paths changing from right to left. Whenever there is a horizontal step (resp. diagonal), a vertical path from $\mathcal{P}$ and one from $\wt{\mathcal{P}}$ cross. This way, $\mathcal{P}$ is uniquely defined by $\wt{\mathcal{P}}$ and vice versa. \\

 Lastly, we want to show that the generating function for all such families of h-paths is indeed given by \eqref{eq:JT}. We can now apply the second part of Remark~\ref{rem:LGV variant} by setting $(a_i,b_i)=(v^{n-i+1},r u^{n-i+1})$ for $1 \leq i \leq n$ and equal to $(1,0)$ for $i>n$. It then suffices to enumerate all paths from $\wt{S}_i$ or $\wt{S}_i^\prime$ respectively to $\wt{E}_j$. 
 Getting from $\wt{S}_i$ to $\wt{E}_j$ requires $i-j$ diagonal steps with weights possibly repeating (this happens when the steps lie on the same diagonal as in the path starting from $\wt{S}_5^\prime$ in Figure~\ref{fig:h-path} for instance). Hence, the generating function of this path is given by $v^{n+1-i}h_{i-j}(\mathbf{x})$ when $1\leq i \leq n$ and by $h_{i-j}(\mathbf{x})$ when $n+1\leq i \leq n+k$. From $\wt{S}_m^{\prime\prime}$ to $\wt{E}_j$, we take $k+1+2n-m-j$ steps with a possibly repeating weight. Thus, the generating function for this segment reads $h_{k+1+2n-m-j}(\mathbf{x})$. From $\wt{S}_i^\prime$ to $\wt{S}_m^{\prime\prime}$, we need $n-i$ steps total, $m-i$ of which are vertical and thus weighted by $\frac{w}{u}$. Recalling that we start with an initial weight of $r u^{n+1-i}$, this gives the weight $r u^{n+1-i}\binom{n-i}{m-i}\left(\frac{w}{u}\right)^{m-i}$ for the segment, amounting to
\[
r u^{n+1-i}\sum\limits_{m}^{n-i}\binom{n-i}{m-i} h_{k+1+2n-m-j}(\x)  \left(\frac{w}{u}\right)^{m-i},
\]
for the whole path from $\wt{S}_i^\prime$ to $\wt{E}_j$ yielding the assertion.

\subsection{Two types of Giambelli determinants}
\label{sec:mixedpaths}

In order to obtain the Giambelli type identity \eqref{eq:Giambelli 2}, we apply a mixed dualisation to the family of $e$-paths. For $1 \leq i \leq n$, define the points
\[
\wh{E}_i = (k+1,-k-2i+1) \qquad \textit{and} \qquad \wh{E}_i^\prime=(k+i,-k-i).
\]
A \deff{family of mixed paths of rank $(n,k)$} is a family of $|I|$ non-intersecting lattice paths for a subset $I \subseteq \{1,2,\ldots,n\}$ with starting points $S_i$ and endpoints $\wh{E}_i$ for $i \in I$. The step set for a point $(a,b)$ is again depending on its position.
\begin{itemize}
\item For $a<k$, the allowed steps are $(1,0)$ and $(0,1)$ with weights $x_{a+b+1}$ and $1$, respectively. 
\item For $a=k$, the allowed steps are $(1,-1)$ and $(0,1)$ with weights $x_{a+b+1}$ and $1$, respectively. 
\item For $a>k$ and $a+b\geq 0$, the allowed steps are $(1,-1)$ and $(0,-1)$ with the weights $x_{a+b+1}$ and $1$, respectively.
\item For $a>k$ and $a+b\leq 0$, the allowed steps are $(-1,-1)$ and $(0,-2)$ with the weights $1$ and $\frac{w}{u}$, respectively.
\end{itemize}
For an example, see Figure~\ref{fig:Giambelli}. The weight of a family of mixed paths is given by the product of the weights of the steps times $\prod_{i \in I}r u^{i} \prod_{l \in [1,n]\setminus I}v^l$, i.e. we have a weight $r u^i$ if $S_i$ is chosen and $v^i$ if not. 

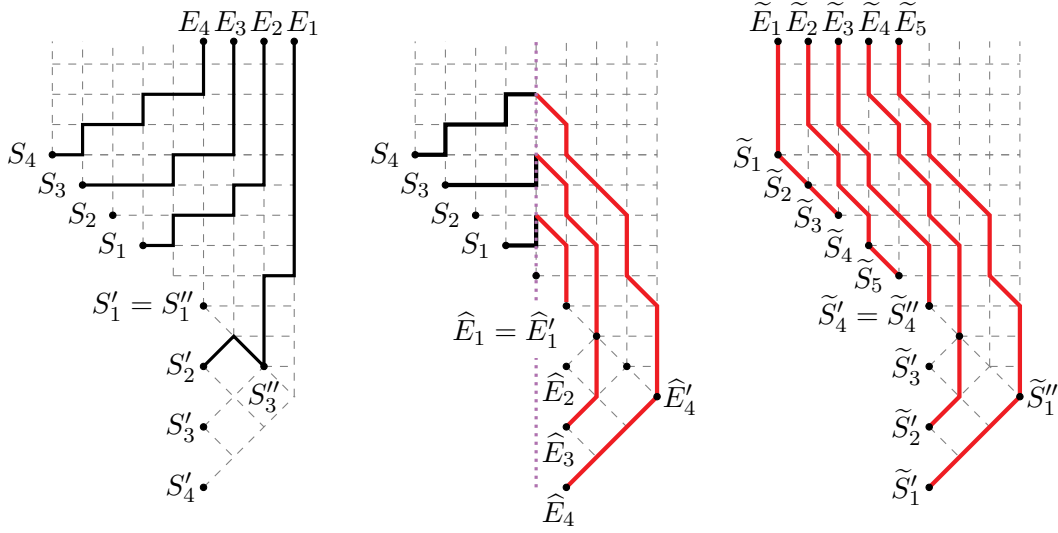
\begin{figure}[h]
\begin{tikzpicture}[scale=.4]

\begin{scope}[xshift=-12cm]
\foreach \i in {-5,...,3}{
		\draw[dashed ,gray] (-1*\i,\i) -- (5,\i);
		\draw[dashed ,gray] (-1*\i,\i) -- (-1*\i,6.75);
}
\foreach \i in {3,...,6}{
		\draw[dashed ,gray] (-3,\i) -- (5,\i);
}
\foreach \i in {1,...,4}{
		\draw[dashed ,gray] (2,-2*\i) -- (6-\i,-4-\i);
		\draw[dashed ,gray] (2,-2*\i) -- (1+\i,-1-\i);
}

\filldraw (0,0) circle (3pt);
\filldraw (-1,1) circle (3pt);
\filldraw (-2,2) circle (3pt);
\filldraw (-3,3) circle (3pt);
\node at (-1,0) {$S_1$};
\node at (-2,1) {$S_2$};
\node at (-3,2) {$S_3$};
\node at (-4,3) {$S_4$};

\filldraw (2,-2) circle (3pt);
\filldraw (2,-4) circle (3pt);
\filldraw (2,-6) circle (3pt);
\filldraw (2,-8) circle (3pt);
\node at (0,-2) {$S_1^\prime = S_1^{\prime\prime}$};
\node at (1.25,-4) {$S_2^\prime$};
\node at (1.25,-6) {$S_3^\prime$};
\node at (1.25,-8) {$S_4^\prime$};

\filldraw (5,6.75) circle (3pt);
\filldraw (4,6.75) circle (3pt);
\filldraw (3,6.75) circle (3pt);
\filldraw (2,6.75) circle (3pt);
\node at (5.3,7.35) {$E_1$};
\node at (4.1,7.35) {$E_2$};
\node at (2.9,7.35) {$E_3$};
\node at (1.7,7.35) {$E_4$};

\filldraw (4,-4) circle (3pt);
\node at (4,-5) {$S_3^{\prime\prime}$};

\draw[very thick] (2,-4) -- (3,-3) -- (4,-4) -- (4,-1) -- (5,-1) -- (5,6.75);
\draw[very thick] (0,0) -- (1,0) -- (1,1) -- (3,1) -- (3,2) -- (4,2) -- (4,6.75);
\draw[very thick] (-2,2) -- (1,2) -- (1,3) -- (2,3) -- (3,3) -- (3,6.75);
\draw[very thick] (-3,3) -- (-2,3) -- (-2,4) -- (0,4) -- (0,5) -- (2,5) -- (2,6.75);
\end{scope}

\foreach \i in {-5,...,3}{
		\draw[dashed ,gray] (-1*\i,\i) -- (5,\i);
		\draw[dashed ,gray] (-1*\i,\i) -- (-1*\i,6.75);
}
\foreach \i in {3,...,6}{
		\draw[dashed ,gray] (-3,\i) -- (5,\i);
}
\foreach \i in {1,...,4}{
		\draw[dashed ,gray] (2,-2*\i) -- (6-\i,-4-\i);
		\draw[dashed ,gray] (2,-2*\i) -- (1+\i,-1-\i);
}

\draw[ultra thick] (0,0) -- (1,0) -- (1,1);
\draw[ultra thick] (-2,2) -- (1,2) -- (1,3);
\draw[ultra thick] (-3,3) -- (-2,3) -- (-2,4) -- (0,4) -- (0,5) -- (1,5);
\draw[ultra thick, Red] (2,-2) -- (2,0) -- (1,1);
\draw[ultra thick, Red] (2,-6) -- (3,-5) -- (3,-3) -- (3,0) -- (2,1) -- (2,2) -- (1,3);
\draw[ultra thick, Red] (2,-8) -- (5,-5) -- (5,-2) -- (4,-1) -- (4,1) -- (2,3) -- (2,4) -- (1,5);

\draw[very thick, Orchid, dotted] (1,-8) -- (1,6.75);

\filldraw (1,-1) circle (3pt)node[below left] {};
\filldraw (0,0) circle (3pt);
\filldraw (-1,1) circle (3pt);
\filldraw (-2,2) circle (3pt);
\filldraw (-3,3) circle (3pt);
\node at (-1,0) {$S_1$};
\node at (-2,1) {$S_2$};
\node at (-3,2) {$S_3$};
\node at (-4,3) {$S_4$};

\filldraw (2,-2) circle (3pt);
\filldraw (2,-4) circle (3pt);
\filldraw (2,-6) circle (3pt);
\filldraw (2,-8) circle (3pt);
\node[fill=white] at (0,-2.8) {$\wh{E}_1=\wh{E}_1^\prime$};
\node at (1.75,-4.75) {$\wh{E}_2$};
\node at (1.75,-6.75) {$\wh{E}_3$};
\node at (1.75,-8.75) {$\wh{E}_4$};

\filldraw (2,-2) circle (3pt);
\filldraw (3,-3) circle (3pt);
\filldraw (4,-4) circle (3pt);
\filldraw (5,-5) circle (3pt);
\node at (5.75,-5) {$\wh{E}_4^\prime$};

\begin{scope}[xshift=12cm]
    \foreach \i in {-5,...,3}{
		\draw[dashed ,gray] (-1*\i,\i) -- (5,\i);
		\draw[dashed ,gray] (-1*\i,\i) -- (-1*\i,6.75);
}
\foreach \i in {3,...,6}{
		\draw[dashed ,gray] (-3,\i) -- (5,\i);
}
\foreach \i in {1,...,4}{
		\draw[dashed ,gray] (2,-2*\i) -- (6-\i,-4-\i);
		\draw[dashed ,gray] (2,-2*\i) -- (1+\i,-1-\i);
}

\draw[ultra thick, Red] (-1,1) -- (-3,3) -- (-3,6.75);
\draw[ultra thick, Red] (1,-1) -- (0,0) -- (0,1) -- (-1,2) -- (-1,3) -- (-2,4) -- (-2,6.75);
\draw[ultra thick, Red] (2,-2) -- (2,0) -- (0,2) -- (0,3) -- (-1,4) -- (-1,6.75);
\draw[ultra thick, Red] (2,-6) -- (3,-5) -- (3,-3) -- (3,0) -- (2,1) -- (2,2) -- (1,3) -- (1,4) -- (0,5) -- (0,6.75);
\draw[ultra thick, Red] (2,-8) -- (5,-5) -- (5,-2) -- (4,-1) -- (4,1) -- (2,3) -- (2,4) -- (1,5) -- (1,6.75);

\filldraw (1,-1) circle (3pt);
\filldraw (0,0) circle (3pt);
\filldraw (-1,1) circle (3pt);
\filldraw (-2,2) circle (3pt);
\filldraw (-3,3) circle (3pt);
\node at (0,-1) {$\wt{S}_5$};
\node at (-1,0) {$\wt{S}_4$};
\node at (-2,1) {$\wt{S}_3$};
\node at (-3,2) {$\wt{S}_2$};
\node at (-4,3) {$\wt{S}_1$};

\filldraw (2,-2) circle (3pt);
\filldraw (2,-4) circle (3pt);
\filldraw (2,-6) circle (3pt);
\filldraw (2,-8) circle (3pt);
\node at (0,-2.25) {$\wt{S}_4^\prime=\wt{S}_4^{\prime\prime}$};
\node at (1.25,-4) {$\wt{S}_3^\prime$};
\node at (1.25,-6) {$\wt{S}_2^\prime$};
\node at (1.25,-8) {$\wt{S}_1^\prime$};

\filldraw (-3,6.75) circle (3pt);
\filldraw (-2,6.75) circle (3pt);
\filldraw (-1,6.75) circle (3pt);
\filldraw (0,6.75) circle (3pt);
\filldraw (1,6.75) circle (3pt);
\node at (-3.4,7.5) {$\wt{E}_1$};
\node at (-2.2,7.5) {$\wt{E}_2$};
\node at (-1,7.5) {$\wt{E}_3$};
\node at (0.2,7.5) {$\wt{E}_4$};
\node at (1.4,7.5) {$\wt{E}_5$};

\filldraw (2,-2) circle (3pt);
\filldraw (3,-3) circle (3pt);
\filldraw (5,-5) circle (3pt);
\node at (5.75,-5) {$\wt{S}_1^{\prime\prime}$};
\end{scope}

\end{tikzpicture}
\caption{\label{fig:Giambelli} On the left: A family of $e$-paths of rank $(4,1)$. On the right: Its dualised family of $h$-paths. In the middle: Its half dualised family of mixed paths.}
\end{figure}

In order to obtain a weight-preserving bijection between the set of families of $e$-paths and the set of families of mixed paths of the same rank, we apply a partial dualisation. In particular, we dualise all $(1,0)$ and $(1,-1)$ steps weakly to the right of the line $x=k$ and fill out the rest of the picture with $(0,-1)$ and $(-1,-1)$ steps accordingly as described before. It is immediate from the construction that for each $1 \leq i \leq n$, we obtain a mixed path starting at $S_i$ exactly if there was an $e$-path starting at $S_i$. Together with the previous observation, this implies that partial dualisation is actually a weight-preserving bijection.

It remains to show that \eqref{eq:Giambelli 2} is the generating function of all families of mixed paths of rank $(n,k)$. First, we calculate the generating function for mixed paths starting at $S_j$ and ending at $\wh{E}_{i}$ for some $1 \leq i,j \leq n$. We split each such path into two parts: the initial part consisting of all steps ending weakly above the line $x+y=0$ and the terminal part. The initial part therefore is a path from $S_j$ to $\wh{E}_m^\prime$ for some $1 \leq m \leq i$. For a given $m$, the generating function of all mixed paths from $S_j$ to $\wh{E}_m^\prime$ is $s_{((m)^{j+k})/(((m-1)^{j+k-1})}(\x)$ where the tableau in the index is a skew shape that looks like a hook reflected along the $x+y=0$ diagonal. Its vertical component corresponds to the $(1,0)$ steps in the $e$-path where $a\leq k$ and its horizontal component to the $(-1,-1)$ steps in the $h$-path where $a\geq k$. By using jeu de taquin, as sketched in Figure~\ref{fig:jeu de taquin}, this can always be transformed bijectively into an actual hook and hence we can write $s_{(m-1|j+k-1)}(\x)$ instead. Note that while the explicit sliding moves depend on the actual semistandard filling, the final shape will be independent of the filling.
 Following the observations of the previous parts, the generating function of paths from $\wh{E}_m^\prime$ to $\wh{E}_i$ is equal to $(\frac{w}{u})^{i-m}\binom{i-1}{i-m}$. The generating function of all paths from $S_j$ to $\wh{E}_{i}$ is therefore
\[
ru^{j} \sum_{m=1}^i  \left(\frac{w}{u}\right)^{i-m}\binom{i-1}{i-m}s_{(m-1|j+k-1)}(\x).
\] 
By using the following lesser known trick or variation of Lindstr\"om--Gessel--Viennot, we obtain the assertion
\begin{multline*}
\sum_{I \subseteq [1,n]} \prod_{l \in [1,n] \setminus I} v^l 
\det_{i,j \in I} \left( ru^{j} \sum_{m=1}^i  \left(\frac{w}{u}\right)^{i-m}\binom{i-1}{i-m}s_{(m-1|j+k-1)}(\x) \right)\\
= \det_{1 \leq i,j \leq n} \left( v^i \delta_{i,j}+ r u^{j} \sum_{m=1}^i  \left(\frac{w}{u}\right)^{i-m}\binom{i-1}{i-m}s_{(m-1|j+k-1)}(\x)\right),
\end{multline*}
where we can interpret the right-hand side as trivial paths from $\wh{E}_i^\prime$ to itself whenever $i$ is not chosen.
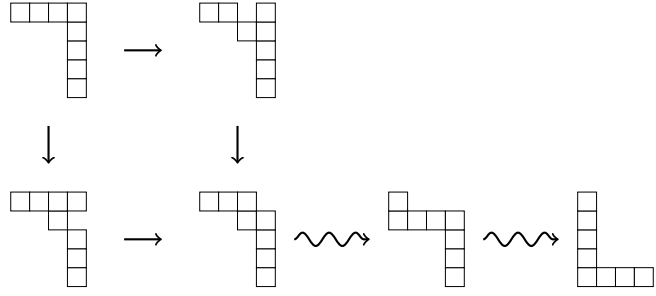
\begin{figure}
\begin{center}
\begin{tikzpicture}[scale=.5]
\tgyoung(0cm,0cm,:::;,:::;,:::;,:::;,;;;;)
\tgyoung(5cm,0cm,:::;,:::;,:::;,::;;,;;:;)
\tgyoung(0cm,-5cm,:::;,:::;,:::;,::;,;;;;)
\tgyoung(5cm,-5cm,:::;,:::;,:::;,::;;,;;;)
\tgyoung(10cm,-5cm,:::;,:::;,:::;,;;;;,;)
\tgyoung(15cm,-5cm,;;;;,;,;,;,;)

\draw[->, thick] (3,1.25) -- (4,1.25);
\draw[->, thick] (3,-3.75) -- (4,-3.75);
\draw[->, thick] (1,-.75) -- (1,-1.75);
\draw[->, thick] (6,-.75) -- (6,-1.75);
\draw[->, thick, snake it] (7.5,-3.75) -- (9.5,-3.75);
\draw[->, thick, snake it] (12.5,-3.75) -- (14.5,-3.75);
\end{tikzpicture}
\caption{\label{fig:jeu de taquin}A sketch of how the shape of $(m)^{j+k})/(((m-1)^{j+k-1})$ is transformed step-wise by jeu de taquin resulting in the shape $(m-1|j+k-1)$ in the example of $m=4$ and $j+k=5$. The two different paths in the first step refer to different values chosen in the explicit SSYT.}
\end{center}
\end{figure}

\subsection{The equality of \eqref{eq:Giambelli 1} and \eqref{eq:Giambelli 2}}
\label{sec:giambellis}

We will now show the equivalence of formulae \eqref{eq:Giambelli 1} and \eqref{eq:Giambelli 2}. For this, we will shift the bounds by $-1$ and multiply the latter with three new matrices such that the determinant remains unchanged which will give the former. First, define the following matrices
\begin{align*}
    A&:= \left[
(-1)^{j-i}v^{j+1}\binom{i}{j}+ r u^{i+1} w^{j-i}s_{(i|j+k)}(\x)
 \right]_{0\leq i,j\leq n-1}\\
    B&:= \left[
 \delta_{i,j}v^{i+1}+ r u^{j+1} \sum_{l} \binom{i}{i-l}
 s_{(l|j+k)}(\x)\left(\frac{w}{u}\right)^{i-l}
 \right]_{0\leq i,j\leq n-1}\\
    C_1&:= \left[(-1)^{i+j}\binom{i}{j}\right]_{0\leq i,j\leq n-1} \\
    C_2&:= \left[\delta_{i,j}\left(\frac{u}{w}\right)^i\right]_{0\leq i,j\leq n-1}\\
    C_3&:= \left[\delta_{i,j}\left(\frac{w}{u}\right)^i\right]_{0\leq i,j\leq n-1}
 \end{align*}
where the first two are just expressions \eqref{eq:Giambelli 1} and \eqref{eq:Giambelli 2} after shifting the indices, respectively. We will show the following relationship
\[
A=C_1\cdot C_2 \cdot B\cdot C_3.
\]
Note that this immediately gives the equivalence of formulae \eqref{eq:Giambelli 1} and \eqref{eq:Giambelli 2} since $\det(C_1)=1$ and $\det(C_2)\cdot\det(C_3)=1$.
 So we perform the following manipulations
\begin{align*}
    C_1\cdot C_2 \cdot B \cdot C_3&=  C_1\cdot \left[\left(\frac{u}{w}\right)^{i}\left(\frac{w}{u}\right)^{j}\left(\delta_{i,j}v^{i+1}+ ru^{j+1} \sum_{l=0}^{i} \binom{i}{i-l}
 s_{(l|j+k)}(\x)\left(\frac{w}{u}\right)^{i-l}\right)  \right]_{i,j} \\
 &= C_1\cdot\left[\delta_{i,j}v^{i+1}+ruw^{j} \sum_{l=0}^{i} \binom{i}{i-l}
 s_{(l|j+k)}(\x)\left(\frac{u}{w}\right)^{l}  \right]_{i,j} \\
 &= \left[\sum^{n-2}_{m=0}(-1)^{i+m}\binom{i}{m}\left(\delta_{m,j}v^{m+1}+ruw^{j} \sum_{l=0}^m \binom{m}{m-l}s_{(l|j+k)}(\x)\left(\frac{u}{w}\right)^{l}\right)\right]_{i,j} \\
 &= \left[(-1)^{j-i}v^{j+1}\binom{i}{j}+ruw^{j}\sum_{m=0}^{i}\sum_{l=0}^m(-1)^{i+m}\binom{i}{m}\binom{m}{m-l}s_{(l|j+k)}(\x)\left(\frac{u}{w}\right)^{l}\right]_{i,j}.
\end{align*}
Since the left summand is already of the desired form, it suffices to calculate the coefficient of $s_{(l|j+k)}(\x)$ for a given $l$. We will show that it is $ru^{i+1} w^{j-i}$ when $l=i$ and $0$ otherwise. First, assume that $l=i$. Then the coefficient indeed is 
\[
ruw^{j}\sum_{m=0}^{i}(-1)^{i+m}\binom{i}{m}\binom{m}{m-i}\left(\frac{u}{w}\right)^{i}=
ruw^{j}(-1)^{2i}\binom{i}{i}\binom{i}{0}\left(\frac{u}{w}\right)^{i}
=ru^{i+1} w^{j-i}.
\]
Now fix an $l\neq i$. Then the coefficient of $s_{(l|j+k)}(\x)$ will be $(-1)^{i} ruw^{j}\left(\frac{u}{w}\right)^{l}\sum_{m=l}^{i}(-1)^{m}\binom{i}{m}\binom{m}{l}$ and so we want to show that
\[
\sum_{m=l}^{i}(-1)^{m}\binom{i}{m}\binom{m}{l}=0.
\]
Using $\binom{i}{m} \binom{m}{l} = \binom{i}{l}\binom{i-l}{i-m}$, we can rewrite the left-hand side of the above expression as
\[
\binom{i}{l} (-1)^{i} \sum_{r=0}^{i-l}(-1)^{r} \binom{i-l}{r} = \binom{i}{l} (-1)^{i}  \left(1+(-1)\right)^{i-l},
\]
which evaluates to $0$ since $i \neq l$.

\section{TSPP tableaux}
\label{sec:TSPP tab}

In the final section of this paper, we introduce another combinatorial interpretation for the symmetric generating function for TSPPs and connect it in special cases to the three dual Littlewood identities.\medskip

We call a partition $\la$ \deff{$k$-asymmetric} if its Frobenius notation is of the form $(a_1,\ldots,a_d|a_1+k,\ldots,a_d+k)$ for positive $k$ and $(a_1-k,\ldots,a_d-k|a_1,\ldots,a_d)$ for negative $k$. For an integer $k\geq -1$ and a $k$-asymmetric partition $\la$, a \deff{TSPP tableau} of shape $\la$ and size $(n,k)$ is a semistandard Young tableau of $\la$ in the symbols $1 < 2 < \cdots < n+k < \ov{1} < \ov{2} <\cdots < \ov{n-1}$ such that
\begin{enumerate}
\item[(C1)] the $(i+1)$-st column has entries at most $\ov{i}$,
\item[(C2)] an entry $\ov{i}$ is only allowed within the first $i$ rows, and
\item[(C3)] the bared entries are row strict, meaning, the entry after $\ov{i}$ must be at least $\ov{i+1}$.
\end{enumerate}
\begin{figure}
\begin{center}
\begin{tikzpicture}
\tyoung(0cm,0cm,112,22{\ov{2}},34,45,56)
\Ycyan
\tgyoung(0cm,0cm,:::;,:::;,:::;,4,56)
\draw[very thick] (0,0) -- (1.5,0) -- (1.5,1) -- (1,1) -- (1,2.5) -- (0,2.5) -- (0,0);

\Ywhite
\tyoung(5cm,0cm,12{\ov{1}}{\ov{3}},23{\ov{2}},3)
\Ycyan
\tgyoung(5cm,-.5cm,,1,23:;,3::;)
\draw[very thick] (5,0) -- (7,0) -- (7,.5) -- (6.5,.5) -- (6.5,1) -- (5.5,1) -- (5.5,1.5) -- (5,1.5) -- (5,0);
\end{tikzpicture}
\caption{\label{fig:TSPP tabs} Two TSPP tableaux of size $(4,2)$ on the left and $(4,-1)$ respectively on the right. For both examples, we coloured the cells counted for the column weights in blue.} 
\end{center}
\end{figure} 
For an example, compare to Figure~\ref{fig:TSPP tabs}. 
Denote by $\TSPPT_\la(n,k)$ the set of TSPP tableaux of shape $\la$ and size $(n,k)$.
We define the weight $\omega(T)$ of a TSPP tableau $T$ as the product of the weights of its entries, where an entry $i$ is weighted by $x_i$ and an entry $\ov{i}$ by $\frac{w}{u}$, times the weight of the columns. For the $j$-th column, we calculate the weight as follows where $1 \leq j \leq n$. If the cell $(i,i)$ is part of $\la$, i.e., $i \leq d$, then the weight is equal to $ru^{1+|\{(i,j) \in \la: i > j+k\}|}$, where $(i,j)$ denotes the $j$-th cell in row $i$. For $d<j\leq n$, the weight is equal to $v^{1+|\{(i,j) \notin \la: i<j\}|}$ with the exception of $k=-1$ and $j=d+1$ in which case we weight this column by $ru+v$.

\begin{thm}
\label{thm:TSPP tableaux}
Let $k\geq -1$ and  $\x=(x_1,\ldots,x_{n+k})$ be a finite family of variables. 
The generating function of all TSPP tableaux of size $(n,k)$ is
\begin{equation}
\label{eq:TSPP tableaux}
\A_{n+1,k}(\x;r,u,v,w) = \sum_{\la} \sum_{T \in \TSPPT_\la(n,k)}\omega(T),
\end{equation}
where the sum is over all $k$-asymmetric partitions $\la \subseteq ((n)^{n+k})$.
\end{thm}
\begin{proof}
Using the expansion of $\A_{n+1,k}(\x;r,u,v,w)$ into families of e-paths of size $(n,k)$ described in Section~\ref{sec:epaths}, it suffices to present a weight preserving bijection between families of e-paths and TSPP tableaux, both of size $(n,k)$. By restricting to $n+k$ variables, it follows that all paths have only $(0,1)$-steps above the $x+y=n+k$ diagonal. Instead of infinite paths, we can therefore restrict ourselves to finite paths with end points $E_j=(k+n+1-j,j-1)$ for $1 \leq j \leq n$.
Given a family $\mc{P}$ of such finite e-paths, we construct a TSPP tableau $T$ as follows; compare to Figure~\ref{fig:bij TSPP tab} for an example.
For each path, label the $(1,0)$ and $(1,-1)$ steps along the path backwards according to the following rule. A $(1,0)$ step is labelled $i$ if it is the $i$-th step of the path, and a $(1,-1)$ step is labelled $\ov{i+h}$ if it is the $(n+k+i)$-th step, where $h$ is the number of $(1,0)$ steps on this path. We read the paths from top to bottom and again within each path from right to left and write the appearing labels in consecutive columns of the resulting tableau.
We first show that the described map is a bijection. Next we show that it is weight preserving up to a simple substitution which does not change the generating function itself.\medskip

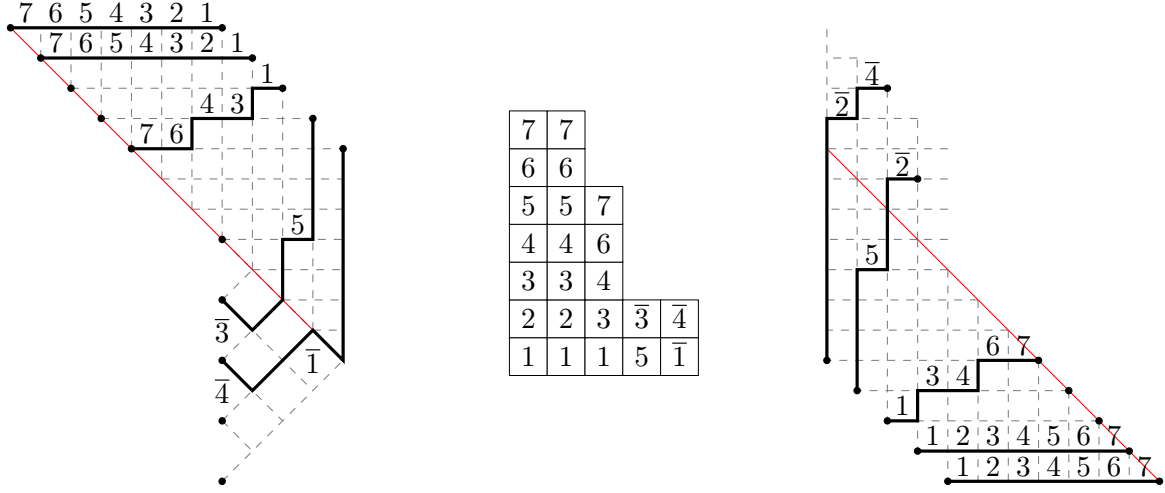
\begin{figure}
\begin{center}
\begin{tikzpicture}[scale=.4]
\foreach \i in {-2,...,4}{
	\draw[dashed ,gray] (-1*\i,\i) -- (-1*\i,4);
}
\foreach \i in {-7,...,-3}{
		\draw[dashed ,gray] (-1*\i,\i) -- (-1*\i,\i+7);
}
\foreach \i in {0,...,4}{
	\draw[dashed ,gray] (-\i,\i) -- (-\i+7,\i);
}
\foreach \i in {1,...,7}{
	\draw[dashed ,gray] (\i,-\i) -- (7,-\i);
}
\foreach \i in {1,...,5}{
		\draw[dashed ,gray] (3,-1-2*\i) -- (8-\i,-6-\i);
		\draw[dashed ,gray] (3,-1-2*\i) -- (2+\i,-2-\i);
}
\draw[red] (-4,4) -- (7,-7);

\filldraw (0,0) circle (3pt);
\filldraw (-1,1) circle (3pt);
\filldraw (-2,2) circle (3pt);
\filldraw (-3,3) circle (3pt);
\filldraw (-4,4) circle (3pt);

\filldraw (3,-3) circle (3pt);
\filldraw (3,-5) circle (3pt);
\filldraw (3,-7) circle (3pt);
\filldraw (3,-9) circle (3pt);
\filldraw (3,-11) circle (3pt);

\filldraw (3,4) circle (3pt);
\filldraw (4,3) circle (3pt);
\filldraw (5,2) circle (3pt);
\filldraw (6,1) circle (3pt);
\filldraw (7,0) circle (3pt);

\draw[very thick] (-4,4) -- (3,4);
\draw[very thick] (-3,3) -- (4,3);
\draw[very thick] (0,0) -- (2,0) -- (2,1) -- (4,1) -- (4,2) -- (5,2);
\draw[very thick] (3,-5) -- (4,-6) -- (5,-5) -- (5,-3) -- (6,-3) -- (6,1);
\draw[very thick] (3,-7) -- (4,-8) -- (5,-7) -- (6,-6) -- (7,-7) -- (7,0);

\node at (2.5,4.5) {1};
\node at (1.5,4.5) {2};
\node at (.5,4.5) {3};
\node at (-.5,4.5) {4};
\node at (-1.5,4.5) {5};
\node at (-2.5,4.5) {6};
\node at (-3.5,4.5) {7};
\node at (3.5,3.5) {1};
\node at (2.5,3.5) {2};
\node at (1.5,3.5) {3};
\node at (.5,3.5) {4};
\node at (-.5,3.5) {5};
\node at (-1.5,3.5) {6};
\node at (-2.5,3.5) {7};
\node at (4.5,2.5) {1};
\node at (3.5,1.5) {3};
\node at (2.5,1.5) {4};
\node at (1.5,.5) {6};
\node at (.5,.5) {7};
\node at (5.5,-2.5) {5};
\node at (3,-6) {$\ov{3}$};
\node at (6,-7) {$\ov{1}$};
\node at (3,-8) {$\ov{4}$};

\begin{scope}[scale=2.5, xshift=5cm, yshift=-3cm]
\tyoung(0cm,0cm,1115{\ov{1}},223{\ov{3}}{\ov{4}},334,446,557,66,77)
\end{scope}

\begin{scope}[xshift=30cm, yshift=-7cm , xscale=-1, yscale=-1]
\foreach \i in {-2,...,4}{
	\draw[dashed ,gray] (-1*\i,\i) -- (-1*\i,4);
}
\foreach \i in {-7,...,-3}{
		\draw[dashed ,gray] (-1*\i,\i) -- (-1*\i,\i+7);
}
\foreach \i in {0,...,4}{
	\draw[dashed ,gray] (-\i,\i) -- (-\i+7,\i);
}
\foreach \i in {1,...,7}{
	\draw[dashed ,gray] (\i,-\i) -- (7,-\i);
}
\foreach \i in {1,...,4}{
	\draw[dashed ,gray] (3+\i,-3-\i) -- (3+\i,-7-\i);
}

\foreach \i in {4,...,7}{
	\draw[dashed ,gray] (3,-\i) -- (\i,-\i);
}
\foreach \i in {8,...,10}{
	\draw[dashed ,gray] (\i-4,-\i) -- (7,-\i);
}
\draw[red] (-4,4) -- (7,-7);

\filldraw (0,0) circle (3pt);
\filldraw (-1,1) circle (3pt);
\filldraw (-2,2) circle (3pt);
\filldraw (-3,3) circle (3pt);
\filldraw (-4,4) circle (3pt);

\filldraw (4,-6) circle (3pt);
\filldraw (5,-9) circle (3pt);

\filldraw (3,4) circle (3pt);
\filldraw (4,3) circle (3pt);
\filldraw (5,2) circle (3pt);
\filldraw (6,1) circle (3pt);
\filldraw (7,0) circle (3pt);

\draw[very thick] (-4,4) -- (3,4);
\draw[very thick] (-3,3) -- (4,3);
\draw[very thick] (0,0) -- (2,0) -- (2,1) -- (4,1) -- (4,2) -- (5,2);
\draw[very thick] (4,-6) -- (5,-6) -- (5,-5) -- (5,-3) -- (6,-3) -- (6,1);
\draw[very thick] (5,-9) -- (6,-9) -- (6,-8) -- (7,-8) -- (7,-7) -- (7,0);

\node at (2.5,3.5) {1};
\node at (1.5,3.5) {2};
\node at (.5,3.5) {3};
\node at (-.5,3.5) {4};
\node at (-1.5,3.5) {5};
\node at (-2.5,3.5) {6};
\node at (-3.5,3.5) {7};
\node at (3.5,2.5) {1};
\node at (2.5,2.5) {2};
\node at (1.5,2.5) {3};
\node at (.5,2.5) {4};
\node at (-.5,2.5) {5};
\node at (-1.5,2.5) {6};
\node at (-2.5,2.5) {7};
\node at (4.5,1.5) {1};
\node at (3.5,.5) {3};
\node at (2.5,.5) {4};
\node at (1.5,-.5) {6};
\node at (.5,-.5) {7};
\node at (5.5,-3.5) {5};
\node at (4.5,-6.5) {$\ov{2}$};
\node at (6.5,-8.5) {$\ov{2}$};
\node at (5.5,-9.5) {$\ov{4}$};
\end{scope}
\end{tikzpicture}
\caption{\label{fig:bij TSPP tab} A family of e-paths of size $(5,2)$ (left), its corresponding TSPP tableau (middle) and the family of e-paths rotated by $180$ degree where the steps originally below the diagonal are ``straightened''.}
\end{center}
\end{figure} 
By rotating $\mc{P}$ by $180$ degrees and ``straightening'' the former $(1,1)$ and $(1,-1)$ steps to $(0,1)$ and $(1,0)$ steps, respectively, as shown in the right of Figure~\ref{fig:bij TSPP tab}, it is immediate that $T$ is an SSYT in the claimed alphabet. Next we check the three conditions for the bared entries are satisfied.
\begin{enumerate}
\item[ad (C1).] By comparing the starting and ending points, we see that the $i$-th path from the top (in the original picture) has at most $n+k+i-1$ steps. If it includes a $(1,-1)$ step, it can have at most $i-1$ many $(1,0)$ and $(1,-1)$ steps. This implies that the highest label is at most $\ov{i-1}$.
\item[ad (C2).]  By definition, an entry $\ov{i}$ in a path is at most the $i$-th labelled entry, i.e., at most in the $i$-th row of $T$.
\item[ad (C3).]  Finally we need to check that bared entries are row strict.
An entry in the $r$-th row and $c$-th column is equal to $\ov{i}$ if it is the $(n+k-h+i)$-th step of the $c$-th path from top. The first $(n+k-h+i)$ steps consist therefore out of $h$ times a $(1,0)$ step, $(n+k-h)$ times the $(0,1)$ step, $r-h$ times a $(1,-1)$ step and $i-r$ times a $(1,1)$ step. Since the path ''starts'' (when going backwards) at $(k+c,n-c)$, its coordinate before the step labelled with $\ov{i}$ is $(k+c-i+1,2r-i-c-k-1)$ and thereafter $(k+c-i,2r-i-c-k)$. Hence the $c$-th and $(c+1)$-st path (counted from top) would cross if they both had the entry $\ov{i}$ in the $r$-th row which is a contradiction.
\end{enumerate}
Next we check that the shape $\la$ of $T$ is $k$-asymmetric.
 Denote by $I$ the index set of starting points $S_i^\prime$ used in $\mc{P}$. First note that the number of labelled steps of a path starting at $S_i^\prime$ and ending at $E_j$ is equal to the number of labelled steps of a path starting at $S_i^{\prime\prime}$ instead, i.e., in order to calculate the number of cells in the $(n+1-j)$-th column of $\la$ we can assume that paths start at $S_i^{\prime\prime}$. Next we rotate the Young diagram $\la$ by 45 degrees counterclockwise and thereby obtain a variation of the Russian notation as shown in Figure~\ref{fig:shape of the tab}. The shape of $\la$ is now defined through a path of $(1,1)$ and $(1,-1)$ steps. We mark each $(1,1)$ step by a square and each $(1,-1)$ by a circle and note the sequence of symbols on a line. By labelling the symbol on the left of the bottom corner of $\la$ by $k$, the square symbols are labelled by the $x$-coordinate of the $S_i$ for $i \in [1,n]\setminus I$ and $S_i^{\prime\prime}$ for $i \in I$.
 To be more precise, for $1 \leq i \leq n$, the $i$-th symbol from the left is a square if and only if $S_{n+1-i}$ is a starting point of $\mc{P}$, i.e., $n+1-i \notin I$, and the $i$-th symbol from the right is a square if and only if $S_{n+1-i}^\prime$ is a starting point of $\mc{P}$, i.e., $n+1-i\in I$.
This implies that the last $n$ steps (on the right) of the boundary of $\la$ are obtained by reflecting the first $n$ steps. The remaining $k$ steps are by definition always marked by a circle, implying the assertion on the shape.\\
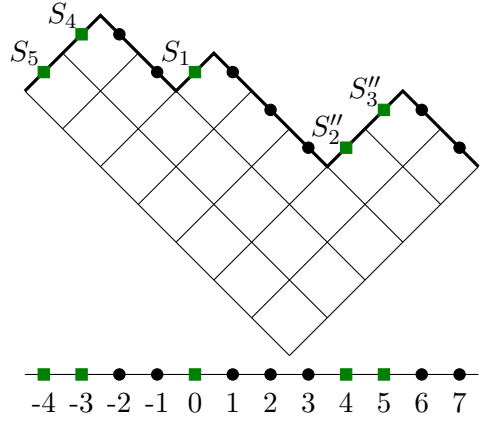
\begin{figure}
\begin{center}
\begin{tikzpicture}[scale=0.5]
\draw (-7,7) -- (0,0);
\draw (-6,8) -- (1,1);
\draw (-3,7) -- (2,2);
\draw (1,5) -- (3,3);
\draw (2,6) -- (4,4);
\draw (0,0) -- (5,5);
\draw (-1,1) -- (4,6);
\draw (-2,2) -- (1,5);
\draw (-3,3) -- (0,6);
\draw (-4,4) -- (-1,7);
\draw (-5,5) -- (-3,7);
\draw (-6,6) -- (-4,8);
\draw[very thick] (-7,7) -- (-5,9) -- (-3,7) -- (-2,8) -- (1,5) -- (3,7) -- (5,5);

\filldraw[green!50!black] (-6.5,7.5) \Square {.15};
\filldraw[green!50!black] (-5.5,8.5) \Square {.15};
\filldraw[black] (-4.5,8.5) circle (.15);
\filldraw[black] (-3.5,7.5) circle (.15);
\filldraw[green!50!black] (-2.5,7.5) \Square {.15};
\filldraw[black] (-1.5,7.5) circle (.15);
\filldraw[black] (-.5,6.5) circle (.15);
\filldraw[black] (.5,5.5) circle (.15);
\filldraw[green!50!black] (1.5,5.5) \Square {.15};
\filldraw[green!50!black] (2.5,6.5) \Square {.15};
\filldraw[black] (3.5,6.5) circle (.15);
\filldraw[black] (4.5,5.5) circle (.15);

\node at (-7,8) {$S_5$};
\node at (-6,9) {$S_4$};
\node at (-3,8) {$S_1$};
\node at (1,6) {$S_2^{\prime\prime}$};
\node at (2,7) {$S_3^{\prime\prime}$};

\draw (-7,-.5) -- (5,-.5);
\filldraw[green!50!black] (-6.5,-.5) \Square {.15};
\filldraw[green!50!black] (-5.5,-.5) \Square {.15};
\filldraw[black] (-4.5,-.5) circle (.15);
\filldraw[black] (-3.5,-.5) circle (.15);
\filldraw[green!50!black] (-2.5,-.5) \Square {.15};
\filldraw[black] (-1.5,-.5) circle (.15);
\filldraw[black] (-.5,-.5) circle (.15);
\filldraw[black] (.5,-.5) circle (.15);
\filldraw[green!50!black] (-6.5,7.5) (1.5,-.5) \Square {.15};
\filldraw[green!50!black] (-6.5,7.5) (2.5,-.5) \Square {.15};
\filldraw[black] (3.5,-.5) circle (.15);
\filldraw[black] (4.5,-.5) circle (.15);

\foreach \i in {-4,...,7}
		\node at (\i-2.5,-1.25) {\i};
\end{tikzpicture}
\caption{\label{fig:shape of the tab} A variation of the Russian notation for the partition $\la$ of Figure~\ref{fig:TSPP tabs} where the steps of the contour are marked with squares and circles respectively.
}
\end{center}
\end{figure}

Since all steps are immediately reversible, we need to show that starting with a TSPP tableau we actually obtain family of e-paths. It is immediate that by reversing the steps, we obtain a family of non-crossing paths which uses the appropriate steps in each region. The above argument on the shape of the tableau implies that if a path starts at $S_i$, no path can start at $S_i^\prime$.\\

Finally we have to check that the bijection is weight preserving up to a twist as explained next. 
By definition, it preserves the $\x$ weight up to the substitution $x_i \mapsto x_{n+k+1-i}$ for all $1 \leq i \leq n+k$. Since the generating function of all e-paths of size $(n,k)$, which is $\A_{n+1,k}$, is invariant under this substitution, we can ignore it for our purpose.
It is immediate that the weight of the $(1,-1)$ steps matches that of bared entries, hence we have to show that the weight coming from the choice of starting positions $S_i$ or $S_i^\prime$ of $\mc{P}$ respectively match the weight of the columns of $T$.
If the $l$-th path of $\mc{P}$ (counted from the top) starts at $S_i$, it has $(i+k+l-1)$ weighted steps and is weighted itself by $ru^i$. As this path corresponds to the $l$-th column of $T$, this column has $(i+k+l-1)$ cells and is therefore by definition weighted by $r u^{1+(i+k+l-1)-(l+k)}=r u^{i}$. On the other hand, if the $l$-th path of $\mc{P}$ starts at $S_i^\prime$ (again counted from the top), it has $(l-i)$ weighted steps and its weight is $v^i$. The corresponding $l$-th column of $T$ has by definition the weight $v^{1+(l-1)-(l-i)}=v^{i}$. Finally we have to deal with the case $k=-1$ and $l=d+1$ (remember that $d$ is the length of the Durfee square of $\la$). In this case, the $l$-th column corresponds to the path starting at $S_1=S_1^\prime$ (both points now have the same coordinates). Instead of having two different starting points with weights $ru$ and $v$, we think of it as one point with weight $(ru+v)$ which completes the proof.
\end{proof}

In the proof above, we obtained a bijection from TSPP tableaux to e-paths by decomposing the tableaux along columns. It is not difficult to see that one obtains analogous bijections to h-paths or mixed paths by decomposing the TSPP tableaux along rows or hooks respectively.\medskip

Remarkably, Theorem~\ref{thm:TSPP tableaux} leads naturally to an interesting relation between the symmetric polynomial $\A_{n,k}$ and the three dual Littlewood identities\footnote{Compare for example to \cite[Ch. I §5, Ex 9, p.79]{Macdonald95}.} which we prove next
\begin{align}
\label{eq:n,1 case}
\A_{n,1}(x_1,\ldots,x_{n};1,1,1,0) &= \sum_{\la: 1\text{-asym}} s_\la(x_1,\ldots,x_n) = \prod_{1\leq i < j \leq n} (1+x_ix_j), \\
\nonumber \A_{n+1,0}(x_1,\ldots,x_{n};1,-1,1,0) &= \sum_{\la: 0\text{-asym}} 
(-1)^{(|\la|+d(\la))/2} s_\la(x_1,\ldots,x_n) \\
\label{eq:n,0 case}
&\phantom{= \sum_{\la: -1\text{-asym}} s_\la(x_1,\ldots,x_n)}
= \prod_{i=1}^n (1-x_i)\prod_{1\leq i < j \leq n} (1-x_ix_j), \\
\label{eq:n,-1 case}
\frac{1}{2}\A_{n+2,-1}(x_1,\ldots,x_{n};1,1,1,0) &= \sum_{\la: -1\text{-asym}} s_\la(x_1,\ldots,x_n) = \prod_{i=1}^n (1+x_i^2)\prod_{1\leq i < j \leq n} (1+x_ix_j),
\end{align}
where $d(\la)$ denotes the size of the Durfee square of $\la$ and we use $k\text{-asym}$ as the abbreviation for $k$-asymmetric. 

First note that for $w=0$, we can not have any bared entries and that $\TSPPT_\la(n,k)$ is therefore equal to the set of all SSYT of shape $\la$ and entries at most $n+k$. This implies immediatly \eqref{eq:n,1 case} since the weight of each TSPP tableau $T$ is $\x^T$ when regarded as an SSYT, and \eqref{eq:n,-1 case} since the weight of each TSSPP tableau $T$ is equal to $2 \x^T$ when regarded as an SSYT. For the last equation \eqref{eq:n,0 case}, the power of $u=(-1)$ counts the number of cells of $\la$ weakly above the main diagonal, which is equal to $(|\la|+d(\la))/2$ and hence proves the claim.

It is not difficult to see that setting $u=-1$ instead of $u=1$ in \eqref{eq:n,1 case} and \eqref{eq:n,-1 case} yields the more standard version of the dual Littlewood identities which have $\prod_{1\leq i < j \leq n} (1-x_ix_j)$ and $\prod_{1\leq i \leq j \leq n} (1-x_ix_j)$ respectively on their right-hand side.\\

Finally, the above suggests to study $\A_{n,k}$ also for negative values of $k$. At least from an enumerative point of view, the ``only'' new interesting case is $k=-1$ for which we observe the following evaluation at $\x=(x_1,\ldots,x_{n-1})=(1,\ldots,1)$, given by
\begin{multline} 
\A_{n,-1}(1,\ldots,1;1,1,1,-1) = 2^{\left\lfloor \frac{n+1}{2} \right\rfloor\left\lfloor \frac{n+2}{2} \right\rfloor-\left\lfloor \frac{n}{2} \right\rfloor} \prod_{i=1}^{\left\lfloor \frac{n+1}{2} \right\rfloor} \frac{(i-1)!}{(n-i)!} \\ \times
\prod_{i \geq 0} 
(3i+1)_{\left\lfloor \frac{n-4i-1}{2} \right\rfloor}
(3i+1)_{\left\lfloor \frac{n-4i-2}{2} \right\rfloor}
\left(2\left\lfloor\frac{n}{2}\right\rfloor -i-\frac{1}{2}\right)_{\left\lfloor \frac{n-4i}{2} \right\rfloor}
\left(2\left\lfloor\frac{n-1}{2}\right\rfloor -i+\frac{1}{2}\right)_{\left\lfloor \frac{n-4i-3}{2} \right\rfloor},
\end{multline}
where we use the Pochhammer symbol defined as $(x)_n:=x(x+1)\cdots (x+n-1)$ for non-negative $n$ and $(x)_n=1$ otherwise. This observation follows from the fact that $\A_{n,k}(x_1,\ldots,x_{n+k-1};1,1,1,-1)$ evaluated at $x_1=\cdots=x_{n+k-1}=1$ is a polynomial in $k$, compare for example to \cite[Equation (21)]{AignerFischer24}. As a consequence of \cite[Theorem 1.2, (3)]{AignerFischer24} this polynomial counts for $d=2k \geq 0$ the number of cyclically symmetric lozenge tilings of a cored hexagon with side lengths $n,n+d,n,n+d,n,n+d$ and is therefore equal to the enumeration of $(n,d)$-AS trapezoids, compare to \cite[Theorem 2.11]{Aigner19}, \cite{Fischer19b}. By setting $k=-1$ we therefore obtain the above formula.

Note, while we presented several combinatorial interpretations for $\A_{n,-1}(\x;1,1,1,-1)$, it would be an interesting question to explore what the corresponding lozenge tiling object would be, i.e., what it means to have a triangular hole of side length $-2$.

\bibliographystyle{abbrvurl}
\bibliography{DetFormulae.bib}
\end{document}